\documentclass[12pt,a4paper]{article}

\usepackage{amsfonts,amssymb,amsmath,amsbsy}
\usepackage{latexsym}
\usepackage{graphicx}
\usepackage{tikz}
\usetikzlibrary{patterns}

\usepackage[colorinlistoftodos, textwidth=2.5cm,textsize=small]{todonotes}

\newtheorem{theo}{Theorem}
\newtheorem{lemma}[theo]{Lemma}
\newtheorem{prop}[theo]{Proposition}
\newtheorem{claim}[theo]{Claim}
\newtheorem{coro}[theo]{Corollary}
\newtheorem{conj}[theo]{Conjecture}

\newtheorem{rem}[theo]{Remark}
\newtheorem{example}[theo]{Example}

\newcommand{\proof}{\noindent{\it Proof: }}
\newcommand{\proofbox}{\hfill \mbox{ $\Box$}\\}
\newcommand{\qed}{\hfill \mbox{ $\Box$}\\}

\newcommand{\R}{{\mathbb R}}

\newcommand{\Z}{{\mathbb Z}}

\author{
K\'aroly J. B\"or\"oczky\thanks{
Alfr\'ed R\'enyi Institute of Mathematics, Hungarian Academy
  of Sciences, Reltanoda u. 13-15, H-1053 Budapest, Hungary, and
Department of Mathematics, Central European University, Nador u 9, H-1051, Budapest, Hungary, 
E-mail: {\tt boroczky.karoly.j@renyi.mta.hu}, 
supported by NKFIH grants 116451, 121649 and 129630}
\and
M\'at\'e Matolcsi   
\thanks{Alfr\'ed R\'enyi Institute of Mathematics, Hungarian Academy
  of Sciences, Reltanoda u. 13-15, H-1053 Budapest, Hungary, E-mail: {\tt matolcsi.mate@renyi.mta.hu},
and Technical University of Budapest, Egry J. u. 1., H-1111 Budapest,
supported by NKFIH grant 109789}
\and
Imre Z. Ruzsa\thanks{Alfr\'ed R\'enyi Institute of Mathematics, Hungarian Academy
  of Sciences, Reltanoda u. 13-15, H-1053 Budapest, Hungary, E-mail: {\tt ruzsa@renyi.hu},
supported by NKFIH grant 109789}
\and
Francisco Santos\thanks{Depto.~de Matem\'aticas, Estad\'{\i}stica y Computaci\'on, Universidad de Cantabria, 39012 Santander, SPAIN. E-mail: {\tt francisco.santos@unican.es}. Supported by grant MTM2017-83750-P of the Spanish Ministry of Science and grant EVF-2015-230 of the Einstein Foundation Berlin} 
\and
Oriol Serra
\thanks{Department of Mathematics, Universitat Polit\`ecnica de Catalunya, and Barcelona Graduate School of Mathematics, Barcelona, Spain. E-mail:  {\tt oriol.serra@upc.edu}. supported by grants MTM2017-82166-P and MDM-2014-0445 of the Spanish Ministry of Science
}}
\title{Triangulations and a discrete Brunn-Minkowski inequality in the plane}

\date{}
\begin{document}

\maketitle

\section{Introduction}

In this paper we write  $A,B$ to denote finite
subsets of $\R^d$,
and $|\cdot|$ stands for their cardinality. We say that $A\subset \R^d$ is $d$--{\it dimensional}  if it is not contained in any affine hyperplane of $\R^d$. Equivalently, the real affine span of $A$ is $\R^d$. 
For objects $X_1,\ldots,X_k$ in $\R^2$,
$[X_1,\ldots,X_k]$ denotes their convex hull. The \emph{lattice generated by $A$} is the  additive subgroup $\Lambda = \Lambda(A)\subset \R^d$ generated by $A-A$, and $A$ is called \emph{saturated} if it satisfies $A=[A]\cap \Lambda (A)$. 

Our starting point are two classical results.
The first one is from the 1950's, due to Kemperman \cite{Kem56}, and popularized by
Freiman \cite{Fre73}:  if $A$ and $B$ are finite nonempty subsets of $\R$,   then
$$
|A+ B|\geq |A|+|B|-1,
$$
with equality if and only
if $A$ and $B$ are arithmetic progressions of the same difference.
The other result,
the Brunn-Minkowski inequality, dates back to the 19th century.
It says that if $X,Y\subset \R^d$ are compact nonempty
sets then
$$
\lambda(X+Y)^{\frac1d}\geq \lambda(X)^{\frac1d}+\lambda(Y)^{\frac1d}
$$
where $\lambda$ stands for the Lebesgue measure. Moreover, provided that $\lambda (X)\lambda(Y)>0$, 
 equality holds if and only if $X$ and $Y$ are convex
homothetic sets.

Various discrete analogues
of the Brunn-Minkowski inequality have been
established in 
Bollob\'as, Leader \cite{BoL91}, Gardner, Gronchi \cite{GaG01},  Green, Tao \cite{GrT06}, Gonz\'alez-Merino, Henze \cite{MeH}, Hern\'andez, Iglesias and Yepes \cite{HeIgYe17},  
Huicochea \cite{Hul18} in any dimension, and  Grynkiewicz, Serra \cite{GrS10} in the planar case.
Most of these  papers use the method
of compression, which changes a
finite set into a set better suited for
sumset estimates, but does not control the convex hull. 

Unfortunately the known analogues are not
as simple in their form as the original Brunn--Minkowski
inequality. For instance,  a formula due to Gardner and Gronchi \cite{GaG01} says that,
 if $A$ is $d$--dimensional, then
\begin{equation}
\label{GraGro}
|A+B|\geq (d!)^{-\frac1d}(|A|-d)^{\frac1d}+|B|^{\frac1d}.
\end{equation}
Concerning the case $A=B$,  Freiman \cite{Fre73} proved that if the dimension of $A$
 is $d$, then
\begin{equation}
\label{Freimanmulti}
|A+A|\ge (d+1)|A|-{d+1 \choose 2}.
\end{equation}
Both estimates are optimal.  In particular, we can not  expect a true discrete analogue of the Brunn--Minkowski  inequality if the notion of volume is replaced by cardinality.

We here conjecture and discuss  a more direct version
of the Brunn--Minkowski inequality where the notion of volume is replaced by the number of full dimensional simplices in a triangulation of the convex hull of the finite set.

For any finite $d$--dimensional set $A\subset \R^d$ we write $T_A$ to denote some triangulation of $A$, by which we mean a triangulation of 
$[A]$ using $A$ as the set of vertices. We denote $|T_A|$ the number of $d$-dimensional simplices in $T_A$. 

In dimension two  the number $|T_A|$ is the same for all triangulations of $A$, so we denote it ${\rm tr}(A)$. More precisely, if  $\Delta_A$ and $\Omega_A$ denote
the number of points of $A$ in the boundary $\partial [A]$
and in the interior ${\rm int}[A]$, respectively, then 
\begin{equation}
\label{Eulerpoints}
{\rm tr}(A)=\Delta_A+2\Omega_A-2=2|A|-\Delta_A-2.
\end{equation}
Therefore in dimension two we can formulate the following  discrete analogue  of the Brunn--Minkowski inequality.

\begin{conj}
\label{ruzsabrunnconj}
If finite $A,B\subset\R^2$ in the plane are not collinear, then
$$
{\rm tr}(A+B)^{\frac12}\geq {\rm tr}(A)^{\frac12}+{\rm tr}(B)^{\frac12}.
$$
\end{conj}

One case where Conjecture~\ref{ruzsabrunnconj} holds with equality is when  $A$ and $B$ are homothetic saturated sets with respect to the same lattice; namely, $A=\Lambda\cap k\cdot P$ and $B=\Lambda \cap m\cdot P$ for a lattice $\Lambda$, polygon $P$ and integers $k,m\ge 1$.  This follows from the original Brunn-Minkowski equality, since  $A+B=\Lambda\cap (k+m)\cdot P$ and the area of any triangle in a suitable triangulation is $\frac12\det\Lambda$. 

We also note that  Conjecture \ref{ruzsabrunnconj}, together with the equality  \eqref{Eulerpoints} and the fact that $\Delta_{A+B}\ge \Delta_A+\Delta_B$,
would imply the following inequality of Gardner and Gronchi  \cite[ Theorem 7.2]{GaG01} for  sets $A$ and $B$
saturated with respect to the same lattice:
$$
|A+B|\ge |A|+|B|+(2|A|-\Delta_A-2)^{1/2}(2|B|-\Delta_B-2)-1.
$$

Unfortunately we have not been able to prove  Conjecture~\ref{ruzsabrunnconj} in full generality. Our main results are the following four cases of it: if $[A]=[B]$ (Theorem~\ref{A=B}), in which case we also determine the conditions for equality in Conjecture~\ref{ruzsabrunnconj}; if $A$ and $B$ differ by one element
(Theorem~\ref{oneextra}); if either $|A|=3$ or $|B|=3$ (Theorem~\ref{triangle-mixed}); and 
if none of  $A$ and $B$ have interior points (Theorem~\ref{convex-position}). Actually, the last two theorems verify a
stronger conjecture (Conjecture~\ref{ruzsabrunnconj-mixed}) discussed below.

We start with the case $[A]=[B]$, which naturally include the case $A=B$.

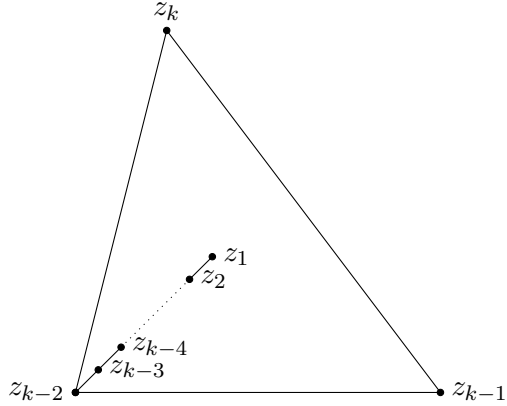
\begin{figure}
\begin{center}
\begin{tikzpicture}[scale=0.6]

\foreach \i in {(0,0), (2,8),(8,0)}
\draw[fill] \i circle(2pt);

\foreach \i in {0,...,2}
\draw[fill] (0.5*\i,0.5*\i) circle(2pt);

\foreach \i in {5,...,6}
\draw[fill] (0.5*\i,0.5*\i) circle(2pt);

\draw (0,0)--(2,8)--(8,0)--(0,0)--(1,1) (2.5,2.5)--(3,3);

\draw[dotted] (1,1)--(2.5,2.5);

\node[below,left] at (0,0) {\small $z_{k-2}$};

\node[below,right] at (0.5,0.5) {\small $z_{k-3}$};
\node[below,right] at (1,1) {\small $z_{k-4}$};
\node[below,right] at (2.5,2.5) {\small $z_{2}$};
\node[below,right] at (3,3) {\small $z_{1}$};
\node[below,right] at (8,0) {\small $z_{k-1}$};
\node[above] at (2,8) {\small $z_{k}$};

\end{tikzpicture}
\end{center}
\caption{ An illustration of case (b) in Theorem \ref{A=B}. }
\end{figure}

\begin{theo}
\label{A=B}
Let  $A,B\subset\R^2$ be finite two dimensional sets. If  $[A]=[B]$ then
Conjecture~\ref{ruzsabrunnconj} holds.
Moreover equality holds
if and only if $A=B$, and
\begin{description}
  \item{(a)} either $A$ is a saturated set, or
  \item{(b)}  $A=\{z_1,\ldots,z_k\}$ for $k\geq 4$, where
  $z_1,\ldots,z_{k-3}\in{\rm int}[z_{k-2},z_{k-1},z_k]$,
  and $z_1,\ldots,z_{k-2}$ are collinear and equally spaced in this order (see Figure~1).
 \end{description}
\end{theo}

Let us mention that Theorem~\ref{A=B} (in fact, its particular case $A=B$) gives a simple proof of the following structure theorem of Freiman \cite{Fre73}
for a planar set with small doubling. We recall that according to (\ref{Freimanmulti}), if
finite $A\subset \R^2$ is two dimensional, then $|A+A|\geq 3|A|-3$ and,  if the dimension of $A$ is at least $3$, then $|A+A|\geq 4|A|-6$.

\begin{coro}[Freiman]
\label{A=Bstability}
Let $A\subset \R^2$ be a fnite two dimensional set and $\varepsilon\in (0,1)$. If $|A|\geq 48/\varepsilon^2$ and
$$
|A+A|\leq (4-\varepsilon)|A|,
$$
then there exists a line $l$ such that $A$ is covered by  at most 
$$
\frac2{\varepsilon}\cdot(1+\frac{32}{|A|\varepsilon^2})
$$
lines parallel to $l$.
\end{coro}

We note that, for $A$ the grid $\{1,\ldots,k\}\times\{1,\ldots,k^2\}$ and  large $k$, 
\begin{equation}
\label{square}
|A+A|\leq (4-\varepsilon)\,|A|,
\end{equation}
with $\varepsilon=\varepsilon_k=\frac{2}{k}$ and $A$ can not be covered by less than $k$ parallel lines.  Therefore
the constant $2$ in the numerator of $\frac2\varepsilon$ is asymptotically optimal in Corollary~\ref{A=Bstability}.

The next case w address is when $A$ and $B$ differ by one element.

\begin{theo}
\label{oneextra}
Let $A\subset \R^2$ be a finite two dimensional set. If $B=A\cup\{b\}$ for some $b\not\in A$ then 
Conjecture~\ref{ruzsabrunnconj} holds.
\end{theo}

For our next results we need the notion of \emph{mixed subdivision} 
(see De Loera, Rambau, Santos \cite{LRF10} for details). 
For finite $d$--dimensional sets $A,B\subset\R^d$  and  triangulations $T_A$ and $T_B$  of $[A]$ and $[B]$, we call a polytopal subdivision $M$ of $[A+B]$ a {\it mixed subdivision}  corresponding to $T_A$ and $T_B$ if
\begin{description}
\item{(i)} every $k$-cell of $M$ is of the form $F+G$ where $F$ is an $i$-simplex of $T_A$ and $G$ is a $j$-simplex of $T_B$ with $i+j=k$;
\item{(i)} for any $d$-simplices $F$ of $T_A$ and $G$ of $T_B$, there is a unique $b\in B$ and a unique $a\in A$ such that $F+b\in M$ and $a+G\in M$.
\end{description}
We write $\|M\|$ to denote the weighted number of $d$-polytopes, where $F+G$ has weight ${i+j \choose i}$ if $F$ is an $i$-simplex of $T_A$, and $G$ is a $j$-simplex of $T_B$ with $i+j=d$. 
In particular, all vertices of $M$ are in $A+B$, and the number of $d$-simplices is $\|M\|$ for any triangulation of $M$ with the same set of vertices   (see e.g.  \cite[Proposition 6.2.11]{LRF10}). 

The main goal of this paper is to  investigate the following problem: For which triangulations $T_A$ and $T_B$ there exists a corresponding mixed subdivision $M$ for $[A+B]$ such that
\begin{equation}
\label{MixedBrunnMinkowski}
\|M\|^{\frac1d}\geq |T_A|^{\frac1d}+|T_B|^{\frac1d}.
\end{equation}

In $\R^2$, we write $M_{11}$ to denote the set of parallelograms  in a mixed subdivision $M$. In this case (\ref{MixedBrunnMinkowski}) is  equivalent to the following stronger version of Conjecture~\ref{ruzsabrunnconj}. 

\begin{conj}
\label{ruzsabrunnconj-mixed}
For every finite two dimensional sets  $A,B\subset\R^2$ there exist triangulations $T_A$ and $T_B$ of $[A]$ and $[B]$ using $A$ and $B$, respectively, as the set of vertices, and a corresponding mixed subdivision $M$ of $[A+B]$ such that
\begin{equation}
\label{MixedBrunnMinkowski2}
|M_{11}|\geq \sqrt{|T_A|\cdot |T_B|}.
\end{equation}
\end{conj}

Conjecture~\ref{ruzsabrunnconj-mixed} offers a geometric and algorithmic approach to prove Conjecture~\ref{ruzsabrunnconj}.

The following example shows that one cannot a priori fix the triangulations $T_A$ and $T_B$ in Conjecture~\ref{ruzsabrunnconj-mixed}:
%
%

\begin{figure}
\label{AA}
\begin{center}
\begin{tikzpicture}[scale=0.5]

\draw[lightgray] (-1,-2) grid (2,1);

\foreach \i in {(0,0), (2,1),(-1,-2)}
\draw[fill] \i circle(2pt);

\draw (0,0)--(2,1)--(-1,-2)--(0,0);
\end{tikzpicture}
\hspace{5mm}
\begin{tikzpicture}[scale=0.5]


\foreach \i in {-5,...,3}
{
	\draw[lightgray]  (\i+1,\i)--(\i+1,\i+1)--(\i+2,\i+1);
	\draw[lightgray] (\i+1,\i+1)--(-5,5)--(\i+1,\i);
	\draw[lightgray] (\i+1,\i+1)--(5,-5)--(\i+1,\i);
\draw[fill] (\i+1,\i) circle(2pt);
\draw[fill] (\i+1,\i+1) circle(2pt);

}
	\draw[lightgray] (5,-5)--(5,4)--(-5,5);
\draw[fill] (-5,5) circle(2pt);
\draw[fill] (5,-5) circle(2pt);
\draw[fill] (5,4) circle (2pt);

\end{tikzpicture}
\end{center}
\caption{ An illustration of the example described in Proposition \ref{counterexample}. }
\end{figure}
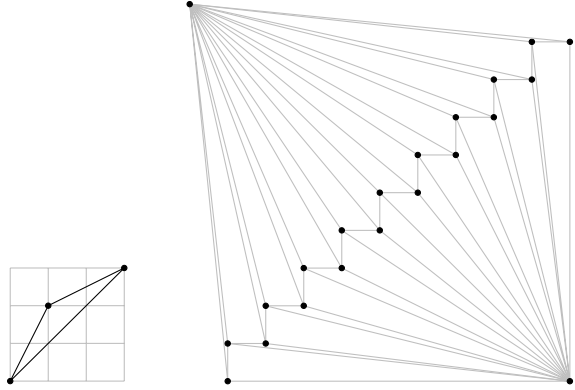

\begin{prop}
\label{counterexample}
Let
$$
A=\{(0,0),(-1,-2),(2,1)\}.
$$
For $k\geq 145$, let
$$
B=\{p,q,l_0,\ldots,l_k,r_0,\ldots,r_{k-1}\},
$$
where 
$p=(-1,k+1)$, $q=(k+1,-1)$, 
$l_i=(i,i)$ for $i=0,\ldots,k$ and $r_i=(i,i+1)$ for $i=0,\ldots,k-1$. 

Let  $T_B$ be the triangulation of $B$ consisting of the triangles 
$$
[p,l_i,r_i],[q,l_i,r_i], \, i=0,\ldots , k-1\;\text{and }\; [p,l_i,r_{i-1}], [q,l_i,r_{i-1}], \; i=1,\ldots, k.
$$ 
Then, no mixed subdivision of $A+B$ corresponding to $T_B$ and any triangulation $T_A$ of $A$ satisfies \eqref{MixedBrunnMinkowski} for $d=2$. 
\end{prop}

 Now Conjecture~\ref{ruzsabrunnconj-mixed} is verified if either $A$ or $B$ has only three elements.

\begin{theo}
\label{triangle-mixed}
If $|B|=3$, then Conjecture~\ref{ruzsabrunnconj-mixed}
 holds for any  finite two dimensional set $A\subset\R^2$.
\end{theo}
{\bf Remark } It follows that if $B$ is the sum of sets of
cardinality three, then Conjecture~\ref{ruzsabrunnconj} holds for any
finite two dimensional set $A\subset\R^2$. For example, if $m\geq 1$ is an integer,
and
 $B=\{(t,s)\in\Z^2:t,s\geq 0\mbox{ and }t+s\leq m\}$,
or $B=\{(t,s)\in\Z^2:|t|,|s|\leq m\mbox{ and }|t+s|\leq m\}$.\\

Conjecture~\ref{ruzsabrunnconj} was verified by B\"or\"oczky, Hoffman \cite{BoH15} if  $A$ and $B$ are in convex position; namely, $A\subset\partial[A]$ and $B\subset\partial[B]$.
Here we
even verify 
Conjecture~\ref{ruzsabrunnconj-mixed} under these conditions.

\begin{theo}
\label{convex-position}
Let  $A,B\subset\R^2$ be finite two dimensional sets. If  $A\subset\partial[A]$ and $B\subset\partial[B]$
then Conjecture~\ref{ruzsabrunnconj-mixed} 
 holds. 
\end{theo}

Part of the reason why we could not verify Conjecture~\ref{ruzsabrunnconj} in general
is that, except for Theorem~\ref{triangle-mixed}, our arguments
 actually prove the inequality ${\rm tr}(A+B) \ge 2 ({\rm tr}(A) + {\rm tr}(B))$, which is stronger than Conjecture~\ref{ruzsabrunnconj}, but which does not hold  for all pairs with $A \subset B$. For example, if $A$ are the nonnegative integer points with sum of coordinates at most $k$ and $B$ is the same with sum of coordinates at most $l$, we have ${\rm tr}(A+B) = (k+l)^2$, ${\rm tr}(A)=k^2$ and ${\rm tr}(B)=l^2$. So we have 
${\rm tr}(A+B) < 2 ({\rm tr}(A) + {\rm tr}(B))$  if $k\neq l$.

Turning to higher dimensions, we note that if $T_A=T_B$, then 
a mixed subdivision satisfying (\ref{MixedBrunnMinkowski}) does exist.  

\begin{theo}
\label{TA=TB}
For a finite  $d$--dimensional set $A\subset\R^d$ and for any triangulation $T_A$ of $[A]$ using $A$ as the set of vertices there exists a corresponding mixed subdivision $M$ of $[A+A]$ such that
$$
\|M\|= 2^d|T_A|.
$$
\end{theo}

Therefore in certain cases, 
mixed subdivisions point to a higher dimensional generalization of Conjecture~\ref{ruzsabrunnconj}. This is specially welcome knowing that, if $d\geq 3$, then the order of the number of $d$-simplices in a triangulation of the convex hull of a finite $A\subset\R^d$ spanning $\R^d$ might be as low as $|A|$ and  as high as $|A|^{\lfloor d/2\rfloor}$ for the same $A$. In particular, one can not assign the number of $d$-simplices as a natural notion of discrete volume if $d\geq 3$.

\section{Proof of Theorem~\ref{A=B}}

We will actually prove that
\begin{equation}
\label{[A]=[B]strong}
{\rm tr}(A+B) \ge 2 {\rm tr}(A) + 2 {\rm tr}(B),
\end{equation}
a stronger inequality than Conjecture~\ref{ruzsabrunnconj}.

For a finite two dimensional set  $X\subset\R^2$, we define
$$
f_X(z)=\left\{
\begin{array}{rl}
1&\mbox{ \ if $z\in\partial[X]$}\\[0.5ex]
2&\mbox{ \ if $z\in{\rm int}\,[X]$}
\end{array}
\right. ,
$$
so that 
$$
{\rm tr}(X)=\left( \sum_{z\in X}f_X(z)\right)-2.
$$

\begin{lemma} Let  $A,B\subset\R^2$   satisfy $[A]=[B]$. Then  inequality \eqref{[A]=[B]strong} holds. 
Moreover, equality  in (\ref{[A]=[B]strong}) yields $A=B$.
\end{lemma}

\begin{proof} Let $T$ be a triangulation of $[A]=[B]$ using the points in $A\cap B$ as vertices. One nice thing about inequality~\eqref{[A]=[B]strong} is that, since it is linear, it is additive over the triangles of $T$. Therefore, it suffices to show that, for each triangle  $t$  of $T$, 
if $A_t=A\cap t$ and $B_t=B\cap t$, then
\begin{equation}
\label{[A]=[B]triangle}
{\rm tr}(A_t+B_t) \ge 2 {\rm tr}(A_t) + 2 {\rm tr}(B_t),
\end{equation}
and  that equality in (\ref{[A]=[B]triangle}) implies that $A_t=B_t$ consists of the three vertices of $t$ alone. Moreover, inequality \eqref{[A]=[B]triangle} is equivalent to
\begin{equation}
\label{[A]=[B]f}
\sum_{p\in A_t+B_t }f_{A_t+B_t}(p)=
\left(\sum_{p\in A_t }f_{A_t}(p)\right)+
\left(\sum_{p\in B_t }f_{B_t}(p)\right)  - 6.
\end{equation}

Let $A_t\cap B_t=\{v_1,v_2,v_3\}$ be the three vertices of the triangle $t=[A_t]=[B_t]$. We claim that if $i,j\in\{1,2,3\}$, $p\in (A_t\cup B_t)\backslash \{v_1,v_2,v_3\}$ and  
$q\in A_t\cup B_t$, then
\begin{equation}
\label{tsums}
v_i+p=v_j+q\mbox{ \ yields \ } v_i=v_j\mbox{ \ and \ } p=q.
\end{equation}
We may assume that $v_i$ is the origin and, to get a contradiction, $v_i\neq v_j$. Then the line $l$ passing through $v_j$ and parallel to the side of $t$ opposite to $v_j$ separates  $t$ and $v_j+t$, and intersects $t$ only in $v_j\neq p$. Since $v_j+q\in v_j+t$, we get the desired contradiction. 

It follows from (\ref{tsums}) that the six points $v_i+v_j$, $1\leq i\leq j\leq 3$, and the points of the form
$v_i+p$, $i=1,2,3$ and $p\in (A_t\cup B_t)\backslash \{v_1,v_2,v_3\}$ are all different. Since the six points $v_i+v_j$, $1\leq i\leq j\leq 3$, belong to $\partial(A_t+B_t)$, we have 
\begin{equation}
\label{tvertices}
\left(\sum_{i,j=1,2,3}f_{A_t+B_t}(v_i+v_j)\right)=
\left(\sum_{i=1}^3f_{A_t}(v_i)\right)+
\left(\sum_{j=1}^3f_{B_t}(v_j)\right) = 6.
\end{equation}
On the other hand, we claim that, if $p\in A_t\backslash \{v_1,v_2,v_3\}$ and $q\in B_t\backslash \{v_1,v_2,v_3\}$,
then
\begin{equation}
\label{tnovertices}
\begin{array}{rcl}
\sum_{j=1}^3f_{A_t+B_t}(p+v_j)&>&2f_{A_t}(p) \\
\sum_{i=1}^3f_{A_t+B_t}(v_i+q)&>&2f_{B_t}(q).
\end{array}
\end{equation}
Indeed, the inequality readily holds if $p\in\partial[A_t]$ and, if  $p\in{\rm int}\,[A_t]$, then $p+v_j\in{\rm int}\,[A_t+B_t]$ for $j=1,2,3$, as well, yielding (\ref{tnovertices}).

By combining \eqref{tvertices} and \eqref{tnovertices} we get \eqref{[A]=[B]f} and in turn \eqref{[A]=[B]strong}. Moreover, \eqref{tnovertices} shows that if equality holds in \eqref{[A]=[B]triangle} then  $A_t=B_t$ and, therefore, if equality holds in \eqref{[A]=[B]strong}, then $A=B$.\qed
\end{proof}

For a finite two dimensional set $A\subset \R^2$ and a triangulation $T$ of $A$ we denote by $A_T$  the union of $A$
and the set of midpoints of the edges of $T$ (see Figure \ref{fig:midpoints}).

\begin{lemma}\label{lem:A=B} Let $A\subset \R^2$ be a finite a finite two dimensional set. The equality 
$$
{\rm tr}(A+A)=4\cdot{\rm tr}(A)
$$
holds  if, and only if, for every triangulation $T$ of $[A]$, we have $A_T=\frac12(A+A)$.
 \end{lemma}
 
 \begin{proof}  Divide each triangle
$t$ of $T$ into four triangles using the vertices of $t$
and the midpoints of the sides of $t$. This way we have
obtained a triangulation of $[A]=[A_T]$ using $A_T$
as the vertex set. Therefore
$$
{\rm tr}(A+A)={\rm tr}(\mbox{$\frac12$}(A+A))
\geq {\rm tr}(A_T)=4\cdot{\rm tr}(A).
$$
Moreover, there is equality if and only if $A_T=\frac12(A+A)$.\qed
 \end{proof}

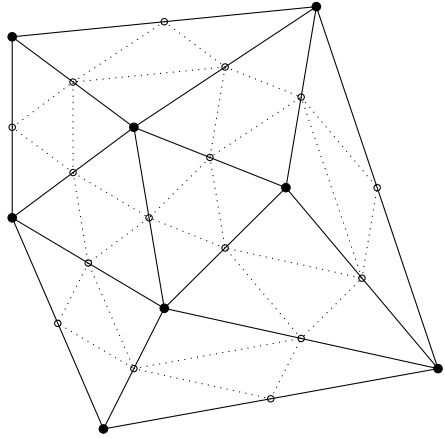
\begin{figure}
\begin{center}
\begin{tikzpicture}[scale=0.4]


\draw (0,7)--(0,13)--(10,14)--(14,2)--(3,0)--(0,7)--(4,10)--(5,4)--(0,7) (0,13)--(4,10)--(10,14)--(9,8)--(14,2)--(5,4)--(9,8)--(4,10) (3,0)--(5,4);

\foreach \i in {(0,7), (0,13), (3,0),  (4,10), (5,4), (9,8), (10,14), (14,2)}
{	
\draw[fill] \i circle(4pt);
}

\foreach \i in {(0,10), (5,13.5), (12,8),(8.5,1),(1.5,3.5),(2,8.5),(4.5,7), (2.5,5.5), (2,11.5), (7,12),(9.5,11),(11.5,5),(9.5,3), (7,6), (6.5,9), (4,2)}
{	
\draw \i circle(3pt);
}

\draw[dotted] (0,10)--(2,11.5)--(2,8.5)--(0,10) (2,8.5)--(4.5,7)--(2.5,5.5)--(2,8.5) (2.5,5.5)--(1.5,3.5)--(4,2)--(2.5,5.5) (4,2)--(9.5,3)--(8.5,1)--(4,2) (9.5,3)--(7,6)--(11.5,5)--(9.5,3) (7,6)--(4.5,7)--(6.5,9)--(7,6) (2,11.5)--(5,13.5)--(7,12)--(2,11.5) (7,12)--(9.5,11)--(6.5,9)--(7,12) (9.5,11)--(12,8)--(11.5,5)--(9.5,11); 

\end{tikzpicture}
\end{center}
\caption{A triangulation and its midpoints.}\label{fig:midpoints}
\end{figure}

We observe that the equation in Lemma \ref{lem:A=B}  is equivalent to Conjecture \ref{ruzsabrunnconj} for the case $A=B$.  Therefore all we have left to prove is that 
${\rm tr}(A+A)=4\cdot{\rm tr}(A)$ if and only if  $A$ is of the form either (a) or (b) in Theorem~\ref{A=B}. The if part is simple.

\begin{lemma} Suppose that either (a) or (b) in Theorem~\ref{A=B} hold for the finite set $A$. Then 
$$
A_T=\frac{1}{2}(A+A).
$$
\end{lemma}

\begin{proof} Suppose first that $A=[A]\cap \Lambda$ for a lattice $\Lambda$. We may assume $\Lambda =\Z^2$. Then clearly the midpoints of sides of every triangulation $T$ of $[A]$  using $A$ as vertex set are precisely the points of  $\frac{1}{2}(A+A)$. 

Next, if we have property (b), then there is a unique  triangulation $T$ of $[A]$ using
$A$ as vertex set. For $1\leq i <j\leq k$, $[z_i,z_j]$
is an edge of $T$, unless $j\leq k-2$,
an hence we have $A_T=\frac{1}{2}(A+A)$  again. \qed
\end{proof}

The next Lemma shows the reverse direction and concludes the proof of Theorem~\ref{A=B}.  

\begin{lemma}  
Let $A\subset \R^2$ be a finite  two dimensional set.  If for every triangulation $T$ of $A$ it holds that 
$$
A_T=\frac{1}{2}(A+A),
$$
then either (a) or (b) from Theorem \ref{A=B} hold. 
\end{lemma}

\begin{proof}    We first prove two simple claims. All throughout we assume that $A_T=\frac{1}{2}(A+A)$ for every triangulation $T$ of $A$.

\begin{claim}\label{claim:line} Let $\ell$ be a line intersecting $A$ in at least two points and $A_{\ell}=A\cap \ell$. If $A_{\ell}+A_{\ell}=(A+A)\cap (\ell+\ell)$ then the points in $A_{\ell}$ form an arithmetic progression. In particular,  the points on each side of the convex hull of $A$ form an arithmetic progression.
\end{claim}

\begin{proof} There is a triangulation $T$ of $A$ which contains the edges defined by consecutive points in $A_{\ell}$.  Since there are $|A_{\ell}|-1$ midpoints of $T$ on $A_{\ell}$,  by the hypothesis of the Lemma and of the Claim,  we have 
$$
|A_{\ell}+A_{\ell}|=|(A+A)\cap (\ell+\ell)|=|A_T\cap \ell|=2|A_{\ell}|-1,
$$
which implies that $A_{\ell}$ consists of an arithmetic progression. \qed
\end{proof}

Call a set of four points of $A$  no three of which collinear an empty quadrangle of $A$ if their convex hull contains no further points of $A$.

\begin{claim}\label{claim:quadrangle} Let $x_1,x_2,x_3,x_4\in A$ form an empty quadrangle of $A$. If they are in convex position  then the four points form a parallelogram. That is, assuming they are listed in clockwise order, we have $x_1+x_3=x_2+x_4$.
\end{claim}

\begin{proof} There are two triangulations of $A$ containing the edges of the convex quadrangle, one of them containing the edge $x_1x_3$ and the other one containing $x_2x_4$. Since $A_T$ cannot depend on the triangulation, the midpoints of these two edges must coincide and therefore $x_1+x_3=x_2+x_4$.\qed \end{proof}

The proof of the Lemma is by induction on $k=|A|$. The Lemma clearly holds if $k=3$. 

Suppose $k=4$. If three of the points are collinear then they are on an edge of the convex hull of $A$ and, by Claim \ref{claim:line},  they form an arithmetic progression. With the fourth one they form a saturated set. If no three of the points are collinear then the four points form an empty quadrangle. If they are in convex position then by Claim \ref{claim:quadrangle} they form a saturated set, otherwise case (b) holds.  

Let $k>4$. Choose a vertex $v$ of the convex hull of $A$ and let $A'=A\setminus \{v\}$. If all points of $A'$ are collinear then by Claim \ref{claim:line} they are in a progresion and, with $v$, they form a saturated set. Suppose that $A'$ is not on o a line. For every triangulation $T'$ of $A$ there is a triangulation $T$ of $A$ containing $T'$. The points in $\frac{1}{2}(A'+A')$ are contained in the convex hull  of $A'$ and, by the condition of the Lemma, coincide with $A'_{T'}$. By induction either (a) or (b) hold for $A'$. We consider the two cases.

{\it Case 1.}   $A'$ is a saturated set. 

{\it Case 1.1.} There is a convex empty quadrangle formed by $v$ and three points of $A'$. Then, by Claim \ref{claim:quadrangle}, $v$ belongs to the lattice generated by $A'$ as well. Moreover, since $A'$ is  convex,  $A$ is also  convex and case (a) holds. 

{\it Case 1.2.} There is no convex empty quadrangle involving $v$ and three points of $A'$. Then it is easily checked that $A'$ has at most one empty convex quadrangle. 

If there is none in $A'$ then, up to an affine transformation, $A'$ consists of the point $(0,1)$ or the two points $(0,\pm 1)$, and the remaining points on the line $y=0$. Then either (i) $v$ belongs to the same line $y=0$, which satisfies the condition of Claim \ref{claim:line}, and all points on that line in $A$ are in arithmetic progression, so that $A$ is a saturated set, or (ii) $A'$ contains only the point $(0,1)$ and $v$ is on the line $x=0$, in which case  Claim \ref{claim:line} yields that the three points of $A$ on that line are in arithmetic progression and $A$ is a saturated set again, or (iii) $A'$ contains only the point $(0,1)$ and $v$  belongs  to none of the two lines containing $A'$ and case (b) holds (see Figure \ref{fig:induction}).

If $A'$ contains one convex empty quadrangle then, up to affinities, $A'$ consists of the four points $(0,0), (1,0), (1,1), (0,1)$ and the remaining ones are on the line $x=y$. Moreover $v$ must belong to the latter line as well and Claim \ref{claim:line} yields that the points on that line are in arithmetic progression and $A$ is a saturated set  (see Figure \ref{fig:induction}). 

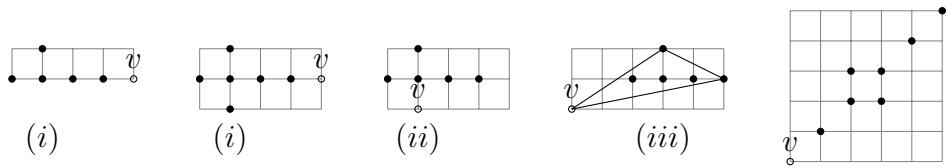
\begin{figure}[h]
\begin{center}
\begin{tikzpicture}[scale=0.4]
\draw[help lines] (-1,0) grid (3,1);
\foreach \i in {(-1,0),(0,0),(1,0),(2,0),(0,1)}
\draw[fill] \i circle(3pt);
\draw (3,0) circle (3pt);
\node[above] at (3,0) {$v$};
\node at (0,-2) {$(i)$};
\end{tikzpicture}
\hspace{3mm}
\begin{tikzpicture}[scale=0.4]
\draw[help lines] (-1,-1) grid (3,1);
\foreach \i in {(-1,0),(0,0),(1,0),(2,0),(0,1), (0,-1)}
\draw[fill] \i circle(3pt);
\draw (3,0) circle (3pt);
\node[above] at (3,0) {$v$};
\node at (0,-2) {$(i)$};

\end{tikzpicture}
\hspace{3mm}
\begin{tikzpicture}[scale=0.4]
\draw[help lines] (-1,-1) grid (3,1);
\foreach \i in {(-1,0),(0,0),(1,0),(2,0),(0,1)}
\draw[fill] \i circle(3pt);
\draw (0,-1) circle (3pt);
\node[above] at (0,-1) {$v$};
\node at (0,-2) {$(ii)$};
\end{tikzpicture}
\hspace{3mm}
\begin{tikzpicture}[scale=0.4]
\draw[help lines] (-3,-1) grid (2,1);
\foreach \i in {(-1,0),(0,0),(1,0),(2,0),(0,1)}
\draw[fill] \i circle(3pt);
\draw (-3,-1) circle (3pt);
\node[above] at (-3,-1) {$v$};
\draw (-3,-1)--(0,1)--(2,0)--(-3,-1);
\node at (0,-2) {$(iii)$};
\end{tikzpicture}
\hspace{3mm}
\begin{tikzpicture}[scale=0.4]
\draw[help lines] (-2,-2) grid (3,3);
\foreach \i in {(0,0),(1,0),(1,1), (0,1), (2,2),(3,3),(-1,-1)}
\draw[fill] \i circle(3pt);
\draw (-2,-2) circle (3pt);
\node[above] at (-2,-2) {$v$};
\end{tikzpicture}
\caption{An illustration of  Case 1.2.}\label{fig:induction}
\end{center}
\end{figure}

{\it Case 2.} $A'$ is as in (b). We may assume that the progression of points of $A'$ lies on the line $x=0$. If $v$ is not on this line then it forms a convex empty quadrangle with two extreme points of the progression and one of the vertices $w$  of the triangle. By Claim \ref{claim:quadrangle}, $v$ must be the point $w+(\pm 1,0)$, which gives a configuration not satisfying the condition of the Lemma. Therefore $v$ lies on the line $x=0$ which satisfies the condition of Claim \ref{claim:line}, so that $v$ belongs to the progression on that line yielding case (b).\qed\end{proof}


\section{Proof of Theorem~\ref{oneextra}}

The inequality between the quadratic and arithmetic means
gives that, if $a,k>0$, then
$$
(4a+2k)^{\frac12}>a^{\frac12}+(a+k)^{\frac12}.
$$
Therefore to prove Theorem~\ref{oneextra}, it is
sufficient the verify the following: Let $B=A\cup\{b\}$
for $b\not\in A$.
\begin{description}
  \item[(*)] If ${\rm tr}(A)=a$ and ${\rm tr}(B)=a+k$,
  then ${\rm tr}(A+B)\geq 4a+2k$.
\end{description}

We fix a triangulation $T$ of $A$, and let $A_T$ be
the union of $A$ and the family of midpoints of the
edges of $T$. It follows by \eqref{Eulerpoints} that
$$
\Delta_{A_T}+2\Omega_{A_T}-2={\rm tr}(A_T)=4a.
$$
To estimate ${\rm tr}(A+B)={\rm tr}(\mbox{$\frac12$}(A+B))$,
we isolate certain subset $V$ of $A$ in a way such that
\begin{equation}
\label{Vcond}
A_T\cap(\mbox{$\frac12$}(V+\{b\}))=\emptyset.  
\end{equation}
Therefore
\begin{eqnarray}
\nonumber
{\rm tr}(A+B)&\geq &
4a+2|\mbox{$\frac12$}(V+\{b\})\cap{\rm int}[B]|+ \\
\label{extrabasic}
&& |\mbox{$\frac12$}(V+\{b\})\cap\partial [B]|+
 |A_T\cap\partial[A]\cap{\rm int}[B]|.
\end{eqnarray}
We distinguish two cases depending on how to define $V$.\\

\noindent{\bf Case 1 } $b\not\in [A]$ 

We say that $x\in [A]$ is visible if $[b,x]\cap[A]=\{x\}$.
In this case $x\in\partial A$. We note that there
are exactly two visible points on $\partial[B]$,
which are on the two supporting
lines to $[A]$ passing through $b$
(see Figure~\ref{fig:case1}). Let $k+1$ be
the number of visible points of $A$, and hence $k\geq 1$.
Now $k-1$ visible points of $A$ lie in ${\rm int}[B]$,
thus \eqref{Eulerpoints} yields that
${\rm tr}(B)= a+k$.
\begin{figure}
\begin{center}
\begin{tikzpicture}[scale=0.4]


\draw (0,7)--(0,13)--(10,14)--(14,2)--(3,0)--(0,7)--(4,10)--(5,4)--(0,7) (0,13)--(4,10)--(10,14)--(9,8)--(14,2)--(5,4)--(9,8)--(4,10) (3,0)--(5,4);

\foreach \i in {(0,7), (0,13), (3,0),  (4,10), (5,4), (9,8), (10,14), (14,2)}
{	
\draw[fill] \i circle(3pt);
}

\foreach \i in {(-3,-3), (0,10), (5,13.5), (12,8),(8.5,1),(1.5,3.5),(2,8.5),(4.5,7), (2.5,5.5), (2,11.5), (7,12),(9.5,11),(11.5,5),(9.5,3), (7,6), (6.5,9), (4,2), (-1.5,2), (-1.5,5), (0,-1.5), (5.5,-0.5)}
{	
\draw \i circle(3pt);
}

\draw[dotted] (0,10)--(2,11.5)--(2,8.5)--(0,10) (2,8.5)--(4.5,7)--(2.5,5.5)--(2,8.5) (2.5,5.5)--(1.5,3.5)--(4,2)--(2.5,5.5) (4,2)--(9.5,3)--(8.5,1)--(4,2) (9.5,3)--(7,6)--(11.5,5)--(9.5,3) (7,6)--(4.5,7)--(6.5,9)--(7,6) (2,11.5)--(5,13.5)--(7,12)--(2,11.5) (7,12)--(9.5,11)--(6.5,9)--(7,12) (9.5,11)--(12,8)--(11.5,5)--(9.5,11) (0,7)--(-3,-3)--(0,13) (3,0)--(-3,-3)--(14,2); 

\node[above] at (-3,-3) {$b$};

\end{tikzpicture}
\end{center}
\caption{An illustration of Case 1.} \label{fig:case1}
\end{figure}
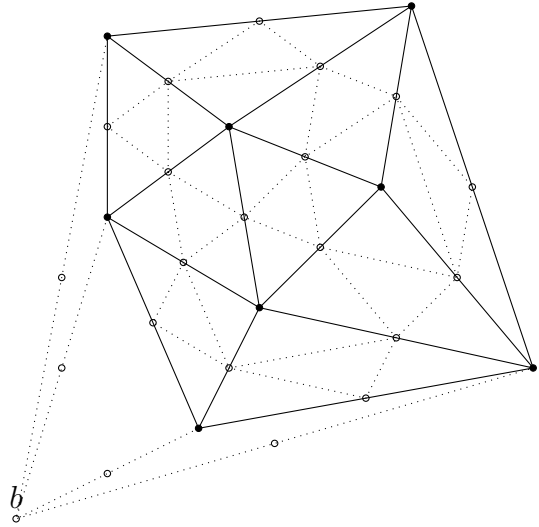
Let $V$ be the set of visible points of $A$.
The condition \eqref{Vcond} is satisfied because
$[A]\cap(\mbox{$\frac12$}(V+\{b\}))=\emptyset$.
We have $|\mbox{$\frac12$}(V+\{b\})|=k+1$,
and $2k-1$ visible points of $A_T$
lie in ${\rm int}[B]$. In particular, $(*)$
follows as \eqref{extrabasic} yields
$$
{\rm tr}(A+B)\geq 4a+2k-1+k+1=4a+3k>4a+2k.
$$

\noindent{\bf Case 2 } $b\in [A]$ 

In this case ${\rm tr}(B)= a+k$ for $k\leq 2$ by
\eqref{Eulerpoints}, and $b$ is contained in 
a triangle $T=[p,q,r]$  of $T$ (see Figure~\ref{fig:case2}).
\begin{figure}
\begin{center}
\begin{tikzpicture}[scale=0.5]


\draw (0,7)--(0,13)--(10,14)--(14,2)--(3,0)--(0,7)--(4,10)--(5,4)--(0,7) (0,13)--(4,10)--(10,14)--(9,8)--(14,2)--(5,4)--(9,8)--(4,10) (3,0)--(5,4);

\foreach \i in {(0,7), (0,13), (3,0),  (4,10), (5,4), (9,8), (10,14), (14,2)}
{	
\draw[fill] \i circle(3pt);
}

\foreach \i in {(0,10), (5,13.5), (12,8),(8.5,1),(1.5,3.5),(2,8.5),(4.5,7), (2.5,5.5), (2,11.5), (7,12),(9.5,11),(11.5,5),(9.5,3), (7,6), (6.5,9), (4,2), (4,12), (7,13), (2,12.5), (4,11)}
{	
\draw \i circle(3pt);
}

\draw[dotted]  (10,14)--(4,12)--(0,13) (4,10)--(4,12);

\node[above] at (4,12) {$b$};
\end{tikzpicture}

\end{center}
\caption{An illustration of Case 2. }\label{fig:case2}
\end{figure}
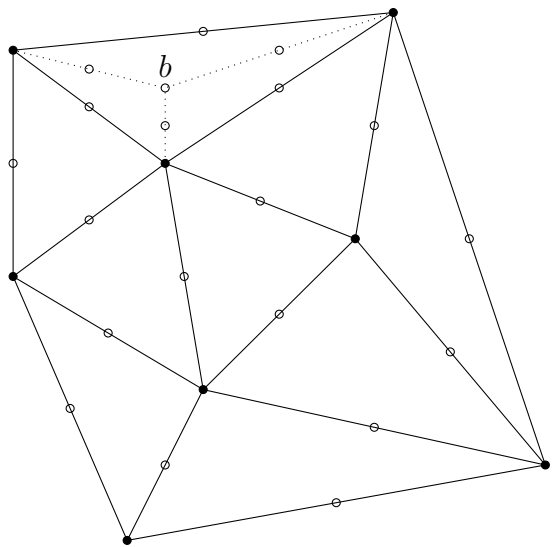
We may assume that $b$ is not contained 
in the sides $[r,p]$ and $[r,q]$ of $T$.
We take $V=\{p,q,r\}$,
which satisfies   \eqref{Vcond}.
Since $\frac12(b+q)\in{\rm int}T\subset{\rm int}[A]$,
\eqref{extrabasic} yields ${\rm tr}(A+B)\geq 4a+4$.
In turn, we conclude Theorem~\ref{oneextra}.\\

\noindent{\bf Remark: } The argument
does not work if we only asssume
that $A\subset B$, because we may
have equality in Conjecture \ref{ruzsabrunnconj} in this case.

\section{Proof of Theorem~\ref{triangle-mixed}}
\label{secmixed-triang}

Let $A\subset \R^2$ be finite and not contained in any line. By a \emph{path} $\sigma$ on $A$ we mean a piecewise linear simple path whose vertices are in $A$, and every point of $A$ in the support of $\sigma$ is a vertex of the path. We write $|\sigma|$ to denote the number of segments forming $\sigma$. We allow the case that $\sigma$ is a point, and in this case $|\sigma|=0$. We say that $\sigma$ is {\it transversal} to a non-zero vector $u$ if every line parallel to $u$ intersects $\sigma$ in at most one point. In this case, the segments in $\sigma$ induce a subdivision of $\sigma+[o,u]$ into $|\sigma|$ parallelograms if $|\sigma|\geq 1$. For the proof of Theorem~\ref{triangle-mixed}  the idea is to find an appropriate set of paths on $A$ with total length at least $\sqrt{|T_A|}$.

First, we explore the possibilities using only one or two paths. We will see in Remark~\ref{onepath} that one path is not enough, but Proposition~\ref{horvert} shows that using two paths $\sigma_1,\sigma_2$ almost does the job.

Observe that for any given non-zero vector $w$, the length of the longest path on $A$ transversal to $w$ equals the number of lines parallel to $w$ intersecting $A$, minus one. 

\begin{rem} 
\label{onepath}
Given pairwise independent vectors $w_1,\ldots,w_n$  let $f(w_1,\ldots,w_n,s)$ be the minimal number such that,
for every finite set $A\subset\R^2$ with ${\rm tr}(A)=s$, there is a $w_i$ and a path on $A$ transversal to $w_i$ of length $f(w_1,\ldots,w_n,s)$. 

For $n=2$,  $f(w_1,w_2,s)\geq\sqrt{s/2}$, with equality   provided that 
$k:=\sqrt{s/2}$ is an integer. An extremal configuration consists of  the points 
$\{iw_1 + jw_2 : i,j\in \{0,\dots,k\} \}$.

For $n=3$,  
$f(w_1,w_2,w_3,s)\geq\sqrt{2s/3}$ and equality holds  provided that $s=6k^2$. 
Assuming without loss of generality that $w_1+w_2+w_3=0$,  an extremal configuration is given by the points of the lattice generated by $w_1,w_2$ in the affine regular hexagon $[\pm kw_1,\pm kw_2,\pm kw_3]$. 
\end{rem}

Let $e_1=(1,0)$ and $e_2=(0,1)$, and let $\sigma_1,\sigma_2$ be piecewise linear paths whose vertices are among the vertices of $A$. We say that the ordered pair 
$(\sigma_1,\sigma_2)$ is a \emph{horizontal-vertical} path if
\begin{description}
\item{(i')} $\sigma_i$ is transversal with respect $e_{3-i}$ (possibly a point), $i=1,2$;
\item{(ii')} the right endpoint $a$ of $\sigma_1$ is the upper endpoint of $\sigma_2$
\item{(iii')} writing $\R_+=\{t\in\R:\, t>0\}$, if $|\sigma_1|,|\sigma_2|>0$, then
$$
\left((\sigma_1\backslash\{a\})+\R_+e_2\right)\cap
\left((\sigma_2\backslash\{a\})+\R_+e_1\right)=\emptyset.
$$
\end{description}
We call $\sigma_1$ the horizontal branch, and $\sigma_2$ the vertical branch,  and $a$ the center.

We observe that if $\sigma'_i$ is the image of $\sigma_i$ by reflection through the line $\R(e_1+e_2)$, then the ordered pair 
$(\sigma'_2,\sigma'_1)$ is also a horizontal-vertical path.

For any polygon $P$ and non-zero vector $u$, we write $F(P,u)$ to denote the face of  $P$ with exterior normal $u$. In particular, $F(P,u)$ is either a side or a vertex.

\begin{prop}
\label{horvert}
For every finite $A\subset\R^2$ not contained in a line, and for every triangulation $T$ of $[A]$ using $A$ as a vertex set, there exists a horizontal-vertical path $(\sigma_1,\sigma_2)$ whose vertices belong to $A$, and satisfies 
$$
|\sigma_1|+|\sigma_2|\geq\sqrt{|T|+1}-\mbox{$\frac12$}.
$$
\end{prop}
\proof Let us write
\begin{eqnarray*}
\xi&=&|F([A],-e_1)\cap F([A],-e_2)|\leq 1\\
\Delta_A'&=&\left|(A\cap\partial[A])\backslash(F([A],-e_1)\cup F([A],-e_2))\right|.
\end{eqnarray*}

By the invariance with respect to reflection through 
the line $\R(e_1+e_2)$, we may assume that
\begin{equation}
\label{hor-vert-size}
|F([A],-e_2)\cap A|\geq |F([A],-e_1)\cap A|.
\end{equation}
We set $\{\langle e_1,p\rangle:\,p\in A\}=\{\alpha_0,\ldots,\alpha_k\}$ with $\alpha_0<\ldots<\alpha_k$, $k\geq 1$. For 
$i=0,\ldots,k$, let $A_i=\{p\in A:\,\langle e_1,p\rangle=\alpha_i\}$, let $x_i=|A_i|$, and
 let $a_i$ be the top most point of $A_i$; namely, $\langle e_2,a_i\rangle$ is maximal. In particular,
$x_0=|F([A],-e_1)\cap A|$.
 For each $i=1,\ldots,k$, we consider the horizontal-vertical path 
$(\sigma_{1i},\sigma_{2i})$ where
$$
\sigma_{1i}=\{[a_0,a_1],\ldots,[a_{i-1},a_i]\},
$$
and  the vertex set of $\sigma_{2i}$ is $A_i$.
In particular, the total length of the  horizontal-vertical path is 
$(\sigma_{1i},\sigma_{2i})$ is
$$
|\sigma_{1i}|+|\sigma_{2i}|=i+x_i-1.
$$
The average length of these paths for $i=1,\ldots,k$ is
$$
\frac{\sum_{i=1}^k(|\sigma_{1i}|+|\sigma_{2i}|)}{k}=
\frac{\sum_{i=1}^{k}(i+x_i-1)}{k}=
\frac{|A|-x_0}{k}+\frac{k}2-\frac12.
$$
We observe that $2|A|=|T|+\Delta_A+2$, according to \eqref{Eulerpoints}, and (\ref{hor-vert-size}) yields
$$
2+\Delta_A-2x_0=2+\Delta_A'+|F([A],-e_2)\cap A|-\xi-x_0\geq \Delta_A'+1.
$$
Therefore we deduce from the inequality between the arithmetic and geometric mean that
\begin{eqnarray}
\nonumber
\frac{\sum_{i=1}^{k-1}(|\sigma_{1i}|+|\sigma_{2i}|)}{k-1}&=&
\frac{2|A|-2x_0}{2k}+\frac{k}2-\frac12\\
\label{horvert-detailed}
&\geq &
\frac{1}{2}\left(\frac{|T|+\Delta_A'+1}{k}+k\right)-\frac12\\
\label{horvert-detailed-sqrt}
&\geq&
\sqrt{|T|+\Delta_A'+1}-\frac12.
\end{eqnarray}
Therefore there exists some  horizontal-vertical path 
$(\sigma_{1i},\sigma_{2i})$ satisfying (\ref{horvert-detailed-sqrt}).
\proofbox

The estimate of Proposition~\ref{horvert} is close to be optimal according to the following example.

\begin{example}
\label{horvert-example}
Let $k\ge 2$ and $t>0$. Let $A'$ be the saturated set with $[A']$ having vertices $(0,0), (0,k), (k-1,0)$ and $(k-1,1)$, and let $A=A'\cup \{(k+t,0)\}$. A triangulation of $A$ has $k^2+k-1$ triangles and every horizontal--vertical path $(\sigma_1,\sigma_2)$ on $A$  has total length
$$
|\sigma_1|+|\sigma_2|\leq k<\sqrt{|T|+2}-\mbox{$\frac12$}.
$$
\qed\end{example}

We next proceed to the proof of Theorem \ref{triangle-mixed} by a similar strategy using three paths.
Let $B=\{v_1,v_2,v_3\}$ and, for $\{i,j,k\}=\{1,2,3\}$ denote by $u_i$ the exterior unit normal to the side $[v_j,v_k]$ of $B$. A set of three paths $(\sigma_1,\sigma_2,\sigma_3)$ meeting at some point $a\in A$ and using the edges of a triangulation $T$ of $A$ is called a {\it proper star} if  the following conditions hold: 
\begin{description}
\item{(i)} $\sigma_i$ is transversal with respect $v_j-v_k$ (possibly $\sigma_i=\{a\}$);
\item{(ii)} $\sigma_i$ has an end point $b_i\in\partial[A]$ such that $u_i$ is an exterior unit normal to $[A]$ at $b_i$, and
$$
\langle a,u_i\rangle=\min\{\langle x,u_i\rangle:x\in\sigma_i\};
$$
\item{(iii)} writing $\R_+=\{t\in\R:\, t>0\}$, if $|\sigma_j|,|\sigma_k|>0$, then
$$
\left((\sigma_j\backslash\{a\})+\R_+(v_k-v_i)\right)\cap
\left((\sigma_k\backslash\{a\})+\R_+(v_j-v_i)\right)=\emptyset.
$$
\end{description}
If the semi-open paths $\sigma_i\backslash\{a\}$, $i=1,2,3$, are all non-empty and pairwise disjoint, then (iii) means that they come around $a$ in the same order as the orientation of the triangle $[v_1,v_2,v_3]$ (see Figure \ref{fig:tristar} for an illustration).

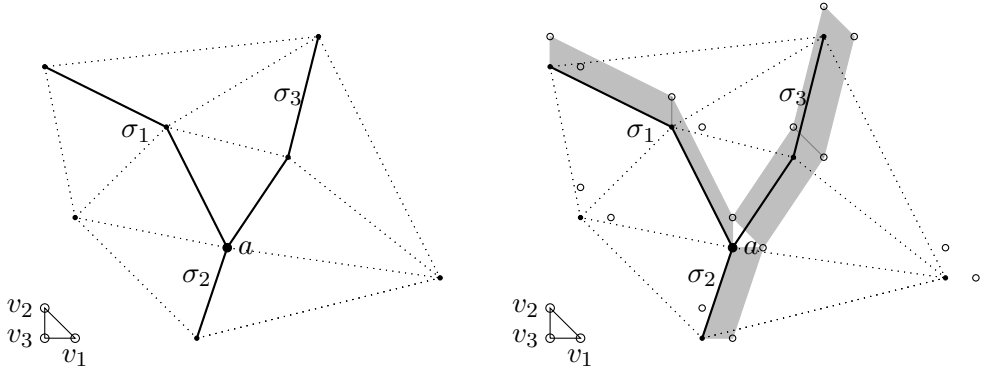
\begin{figure}[h]
	\begin{center}

		\begin{tikzpicture}[scale=0.4]
		
		
		
		\draw (0,0)--(1,0)--(0,1)--(0,0);
		
		\foreach \i in {(0,9), (1,4), (5,-0), (4,7), (6,3), (9,10), (8,6), (13,2)}
		{
			\draw[fill] \i circle(2pt);
		}
		
		\foreach \i in {(0,0), (1,0), (0,1)}
		\draw \i circle (4pt);
		
		\draw[fill] (6,3) circle (4pt);
		\node[below,right] at (6,3) {\small $a$};
		\node[left] at (0,0) {\small $v_3$};
		\node[below] at (1,0) {\small $v_1$};
		\node[left] at (0,1) {\small $v_2$};		
		
		\draw[dotted] (0,9)--(1,4)--(5,-0)--(13,2)--(9,10)--(0,9) (4,7)--(6,3)--(8,6)--(4,7);
		\draw[dotted] (0,9)--(4,7)--(1,4)--(6,3)--(5,-0) (6,3)--(13,2) (5,0)--(13,2)--(8,6)--(9,10)--(4,7);  (4,7)--(6,3)--(8,6)--(4,7);
		
		\draw[dotted] (0,9)--(1,4)--(5,-0)--(13,2)--(9,10)--(0,9) (4,7)--(6,3)--(8,6)--(4,7);
		\draw[dotted] (0,9)--(4,7)--(1,4)--(6,3)--(5,-0) (6,3)--(13,2) (5,0)--(13,2)--(8,6)--(9,10)--(4,7);  (4,7)--(6,3)--(8,6)--(4,7);
		
		\draw[ thick] (0,9)--(4,7)--(6,3)--(5,0) (6,3)--(8,6)--(9,10);
		
		\node at (3,7-0.2) {\small $\sigma_1$};
		\node at (5,2) {\small $\sigma_2$};
		\node at (8,8) {\small $\sigma_3$};
		\end{tikzpicture}
		\hspace{5mm}
				\begin{tikzpicture}[scale=0.4]
		
		
				\draw[fill,lightgray] (0,9)--(0,10)--(4,8)--(4,7)--(0,9);
			\draw[fill,lightgray]  (4,8)--(4,7)--(6,3)--(6,4)--(4,8);
			\draw[fill,lightgray]   (6,3)--(7,3)--(6,0)--(5,0)--(6,3);
			\draw[fill,lightgray]   (6,4)--(7,3)--(9,6)--(8,7)--(6,4);
			\draw[fill,lightgray]    (9,6)--(8,7)--(9,11)--(10,10)--(9,6);
			\draw[gray] (4,7)--(4,8) (8,7)--(9,6);
		
		\draw (0,0)--(1,0)--(0,1)--(0,0);
		
		\foreach \i in {(0,9), (1,4), (5,-0), (4,7), (6,3), (9,10), (8,6), (13,2)}
		{
			\draw \i+(0,1) circle (3pt);
			\draw \i+(1,0) circle (3pt);
			\draw[fill] \i circle(2pt);
		}
		
		\foreach \i in {(0,0), (1,0), (0,1)}
		\draw \i circle (4pt);
		
		\draw[fill] (6,3) circle (4pt);
		\node[below,right] at (6,3) {\small $a$};
		\node[left] at (0,0) {\small $v_3$};
		\node[below] at (1,0) {\small $v_1$};
		\node[left] at (0,1) {\small $v_2$};		
		
		\draw[dotted] (0,9)--(1,4)--(5,-0)--(13,2)--(9,10)--(0,9) (4,7)--(6,3)--(8,6)--(4,7);
		\draw[dotted] (0,9)--(4,7)--(1,4)--(6,3)--(5,-0) (6,3)--(13,2) (5,0)--(13,2)--(8,6)--(9,10)--(4,7);  (4,7)--(6,3)--(8,6)--(4,7);
		
		\draw[dotted] (0,9)--(1,4)--(5,-0)--(13,2)--(9,10)--(0,9) (4,7)--(6,3)--(8,6)--(4,7);
		\draw[dotted] (0,9)--(4,7)--(1,4)--(6,3)--(5,-0) (6,3)--(13,2) (5,0)--(13,2)--(8,6)--(9,10)--(4,7);  (4,7)--(6,3)--(8,6)--(4,7);
		
		\draw[ thick] (0,9)--(4,7)--(6,3)--(5,0) (6,3)--(8,6)--(9,10);
		
		\node at (3,7-0.2) {\small $\sigma_1$};
		\node at (5,2) {\small $\sigma_2$};
		\node at (8,8) {\small $\sigma_3$};
		\end{tikzpicture}
		
	\end{center}
\caption{A proper star  with respect to $v_1,v_2,v_3$ centered at $a$. On the right, paralellograms based on the proper star}
\label{fig:tristar}
\end{figure}

 The next Lemma shows  how to construct an appropriate  mixed subdivision of $A+B$ using a proper star. 

\begin{lemma}
\label{proper-mixed}
Given a proper star with rays $\sigma_1,\sigma_2,\sigma_3$ such that $|\sigma_1|+|\sigma_2|+|\sigma_3|>0$, there exists a mixed subdivision $M$ for $A+B$ satisfying
$$
M_{11}=|\sigma_1|+|\sigma_2|+|\sigma_3|.
$$
\end{lemma}

\proof  We may assume that $|\sigma_1|>0$ and $v_3=o$.
We partition the triangles of $T_A$ into three subsets $\Sigma_1,\Sigma_2,\Sigma_3$ (some of them might be empty). The idea is that if the semi-open paths $\sigma_i\backslash\{a\}$, $i=1,2,3$, are all non-empty and pairwise disjoint and $\{i,j,k\}=\{1,2,3\}$, then $\Sigma_i$ consists of the triangles cut off by $\sigma_j\cup\sigma_k$.

 A triangle $\tau$ of $T_A$ is in $\Sigma_1$ if and only if there exists a  $p\in({\rm int}\,\tau)\backslash(a+\R v_1)$ such that
$$
|(p-\R_+v_1)\cap\sigma_2|+|(p-\R_+v_1)\cap\sigma_3|
$$ 
is finite and odd. Similarly, $\tau\in T_A$ is in $\Sigma_2$ if and only if there exists a  $p\in{\rm int}\,\tau$ such that
$$
|(p-\R_+v_2)\cap\sigma_1|+|(p-\R_+v_2)\cap\sigma_3|
$$   
is finite and odd. The rest of the triangles of $T_A$ form $\Sigma_3$.

The triangles of the mixed subdivision $M$ are as follows. If $\tau\in\Sigma_i$, then the corresponding triangle in $M$ is $\tau+v_i$. In addition, $[B]+a$ is in $M$. For the parallelograms, let $\{i,j,k\}=\{1,2,3\}$. If $e$ is an edge of $\sigma_i$, then $e+[v_j,v_k]$ is in $M$. \proofbox

For the rest of the section, we fix finite $A\subset\R^2$ and $B=\{v_1,v_2,v_3\}\subset\R^2$ such that both of them spans $\R^2$ affinely,  and confirm Conjecture~\ref{ruzsabrunnconj-mixed} in this case. 

The following statement is a simple consequence of the definition of a proper star.

\begin{lemma}
\label{improve-hor-vert}
Assuming $B=\{v_1,v_2,v_3\}$ with $v_1=(1,0)=-u_1$, $v_2=(0,1)=-u_2$ and $v_3=(0,0)$, and hence $u_3=(\frac1{\sqrt{2}},\frac1{\sqrt{2}})$, if $(\sigma_1,\sigma_2)$ is a
 horizontal-vertical path for $A$ centered at $a\in A$, then 
\begin{itemize}
\item there exists  a proper star  $(\sigma'_1,\sigma'_2,\sigma'_3)$ centered at $a$ such that  $\sigma_1\subset\sigma'_1$, 
$\sigma_2\subset \sigma'_2$,
\item  if in addition $a\not\in F([A],u_3)$, then $|\sigma'_3|\geq 1$. 
\end{itemize}
\end{lemma}

\noindent{\bf Proof of Theorem~\ref{triangle-mixed} }
We may assume that $B=\{v_1,v_2,v_3\}$ with $v_1=(1,0)=-u_1$, $v_2=(0,1)=-u_2$ and $v_3=(0,0)$, and hence $u_3=(\frac1{\sqrt{2}},\frac1{\sqrt{2}})$. In addition, we may assume that
$$
|F([A],-u_2)\cap A|\geq |F([A],-u_1)\cap A|.
$$
Using the notation of the proof of (\ref{horvert-detailed}), we set 
$\{\langle u_1,p\rangle:\,p\in A\}=\{\alpha_0,\ldots,\alpha_k\}$ with $\alpha_0<\ldots<\alpha_k$,
and $\Delta_A'=|(A\cap \partial[A])\backslash(F([A],-u_1)\cup F([A],-u_2))|$. For 
$i=0,\ldots,k$, let $A_i=\{p\in A:\,\langle u_1,p\rangle=\alpha_i\}$, let $x_i=|A_i|$ and
 let $a_i$ be the top most point of $A_i$; namely, $\langle u_2,a_i\rangle$ is maximal.
According to (\ref{horvert-detailed}) and  (\ref{horvert-detailed-sqrt}), we have
\begin{equation}
\label{horvert-detailedaverage}
\frac{\sum_{i=1}^{k}(i+x_i-1)}{k}\geq
\frac{|T_A|+\Delta_A'+1}{2k}+\frac{k}2-\frac12
\geq \sqrt{|T_A|+1}-\frac12.
\end{equation}
Let $I$ be the set of all $i\in \{1,\ldots,k\}$ such that
\begin{equation}
\label{ixilower}
i+x_i-1\geq \left\lceil \frac{|T_A|+\Delta_A'+1}{2k}+\frac{k}2-\frac12\right\rceil=\xi.
\end{equation}
Since  $\xi\geq \sqrt{|T_A|+1}-\frac12$,
if strict inequality holds  for some $i$ in (\ref{ixilower}), then we have a required proper star by
 Lemma~\ref{improve-hor-vert}. Thus we assume that $i+x_i-1=\xi$ for $i\in I$.

Let $\theta=|I|$. Since $i+x_i-1\leq \xi-1$ if $i\not\in I$, we have
$$
\xi-\frac{\sum_{i=1}^{k}(i+x_i-1)}{k}\geq \frac{k-\theta}{k}.
$$
We deduce from (\ref{horvert-detailedaverage}) that if $i\in I$, then
$$
i+x_i-1\geq\frac{|T_A|+\Delta_A'+1}{2k}+\frac{k}2-\frac12+\frac{k-\theta}{k}=
\frac{|T_A|+\Delta_A'+1}{2k}+\frac{k}2+\frac12-\frac{\theta}{k}.
$$
If $i\in I$ and $a_i\not\in F([A],u_3)$, then $\xi\geq \sqrt{|T_A|+1}-\frac12$ and  Lemma~\ref{improve-hor-vert} yields the
existence of a required proper star. Therefore we may assume that 
$a_i\in F([A],u_3)$ for $i\in I$. Since  $|F([A],u_3)\cap F([A],-u_2))|\leq 1$, we deduce that
\begin{equation}
\label{tminusepsilonthetak}
\theta\leq \max\{\Delta_A'+1,k\}.
\end{equation}
Therefore if  $i\in I$, then we conclude using the inequality  betwen the arightmetic and the geometric mean at the last inequality that
$$
i+x_i-1\geq \frac{|T_A|+\theta}{2k}+\frac{k}2+\frac12-\frac{\theta}{k}\geq 
\frac{|T_A|}{2k}+\frac{k}2+\frac12-\frac{\theta}{2k}\geq\sqrt{|T_A|}. \mbox{ \ }\proofbox
$$

\section{Proof of Theorem~\ref{convex-position}}
\label{secconvex-position}

We assume in this section that there are no points of $A$ (resp. $B$) in the interior of $[A]$, (resp. $[B]$).

Recall that $\Delta_X$ denotes the number of points of $X$ in the boundary of $[X]$.  
It is easy to check that $\Delta_{A+B}$ has at least as many points
as $\Delta_A$ and $\Delta_B$ together, that is:
\[
\Delta_{A+B} \ge \Delta_A  + \Delta_B = {\rm tr}(A) + {\rm tr}(B) +4
\]
As a motivation for the proof, we note that Conjecture~\ref{ruzsabrunnconj} follows if  the number $\Omega_{A+B}$ of  points of $A+B$
in ${\rm int}([A+B])$ is at least
\[
\frac{{\rm tr}(A) + {\rm tr}(B) - 2}2 = \frac{\Delta_A + \Delta_B} 2  - 3 .
\]
Naturally we aim at the stronger Conjecture~\ref{ruzsabrunnconj-mixed}.
Given Theorem~\ref{triangle-mixed}, Theorem~\ref{convex-position}
follows if $A$ and $B$ being in convex position and $|A|,|B|\geq 4$ yield that
there exists a mixed subdivision of $A + B$ satisfying
\begin{equation}
\label{convex-position-parallelograms}
|M_{11}|\geq \frac{{\rm tr}(A)+{\rm tr}(B)}2.
\end{equation}

Throughout the proof we  assume that $[B]$ has at most as many vertices as $[A]$ and $v$ denotes  a unit vector (which we assume pointing upwards) not parallel to any side of $[A+B]$. We denote by  $a_0$ and $a_1$  the leftmost and rightmost vertex of $[A]$ and by $b_0$ and $b_1$   the leftmost and rightmost vertex of $[B]$.

To prove (\ref{convex-position-parallelograms}),
we say that $A$ and $B$ form a \emph{strange pair}  if   $[B]$ is a triangle and the three exterior normals to $[B]$ are also exterior normals of edges of $[A]$.

We will use that, for   $t,s\geq 1$,
\begin{equation}
\label{tsbig}
ts\geq t+s-1.
\end{equation}

\noindent {\bf Case 1 } $A$ and $B$ are not strange pairs. 

We choose a unit vector $v$ as above   in the following way: if
$B$ is a triangle, then the upper arc of $\partial [B]$ is a side such that $[A]$ has no side with same exterior unit normal; if $[B]$ has at least four sides, then
the two  supporting lines of $[B]$ parallel to $v$ touch at non-consecutive vertices of $[B]$. For the existence of the latter pair of supporting lines, we note that while continuously rotating $[B]$, the number of ��upper - lower vertices�� changes by either zero or two units at a time when a side of $[B]$ is parallel to $v$, and  after rotation by $\pi$ it changes to its opposite. Hence, at some position that difference is zero or one which implies, since $[B]$ has at least four vertices, that at that position there is at least one upper and one lower vertex, as required. 

\begin{claim} One of the two following statements hold:
\begin{equation}
\label{convex-position-case1}
\begin{array}{rl}
&\left|\Big((A + b_0) \cup (a_1 + B)\Big)\cap{\rm int}[A+B]\right|\geq \frac{\Delta_A + \Delta_B} 2  - 3,  \mbox{ or } \\[3mm]
&\left|\Big((a_0 + B) \cup (A + b_1)\Big)\cap{\rm int}[A+B]\right|\geq \frac{\Delta_A + \Delta_B} 2  - 3.
\end{array}
\end{equation}
\end{claim}

\begin{proof}
We may assume that $b_1=a_0=o$ (see Fig. \ref{fig:a0+B}).  Observe first that the only repetitions�� 
$x+b_0 = a_1 +y$ or $x+b_1 = a_0 +y$ in these configurations are the points $a_1+b_0$ and $a_0+b_1$ (which are interior to $[A+B]$ by our hypothesis). To prove (\ref{convex-position-case1}), we verify first that
\begin{description}
\item{(i)} for every $x\in A\setminus \{a_0,a_1\}$ except perhaps two of them, at least one of $x + b_0$ or $x + b_1$ is interior in $A+B$,
\item{(ii)} for every $y\in B\setminus \{b_0,b_1\}$ except perhaps two of them, at least one of  $a_0 + y$ or $a_1+y$ is interior in $A+B$.
\end{description}

\begin{figure}[h]
\begin{center}

	\begin{tikzpicture}[scale=0.5]
	

\draw (-6,-6)--(-6,4) (0,-6)--(0,4) (4,-6)--(4,4);
	
	\foreach \i in {(0, 0),  (-2,2),(-5,1),(-6,-2),(-5,-4),(-1,-3)}
	\draw[fill] \i circle(2pt);
	
	\draw[fill, lightgray] (0, 0)--(-2,2)--(-5,1)--(-6,-2)--(-5,-4)--(-1,-3)--(0,0);
	\draw (0, 0)--(-2,2)--(-5,1)--(-6,-2)--(-5,-4)--(-1,-3)--(0,0);
	
	\foreach \i in {(0,0),(2,2),(4,1),(2,-2)}
	\draw[fill] \i circle(2pt);
	
	\draw[fill,lightgray] (0,0)--(2,2)--(4,1)--(2,-2)--(0,0);
		\draw (0,0)--(2,2)--(4,1)--(2,-2)--(0,0);
	
		\foreach \i in {(-1, 2), (0, 0), (-3, -1), (1, -5), (-1, -3), (-6, -2), (2, 3), (-2,
			2), (-3, 3), (0, 4), (1, -1), (-4, 0), (-5, 1), (2, -2), (-5, -4), (-4,
			-4), (2, 2), (3, -2), (-3, -2), (-2, -1), (-3, -6), (4, 1)}
	\draw \i circle(2pt);
	
	\draw (-5,1)--(-3,3)--(0,4)--(2,3)--(4,1)--(3,-2)--(1,-5)--(-3,-6)--(-5,-4);

	\node[left] at (-6,-2) {$a_0$};
	\node[below] at (0,-6) {$a_1=b_0$};
	\node[right] at (4,1) {$b_1$};
	\node at (-2.5,0) {$b_0+A$};
	\node at (2,0.3) {$B+a_1$};
	\end{tikzpicture}
	\hspace{2mm}
		\begin{tikzpicture}[scale=0.55]
	
	
	\draw (-6,-6)--(-6,4) (0,-6)--(0,4) (4,-6)--(4,4);
	
	\foreach \i in {(0, 0),  (-2,2),(-5,1),(-6,-2),(-5,-4),(-1,-3)}
	\draw[fill] \i circle(2pt);
	
	\draw[fill, lightgray] (-2,-1)--(-1,2)--(2,3)--(4,1)--(3,-2)--(-1,-3)--(-2,-1);
	\draw (-2,-1)--(-1,2)--(2,3)--(4,1)--(3,-2)--(-1,-3)--(-2,-1);
	
	\foreach \i in {(0,0),(2,2),(4,1),(2,-2)}
	\draw[fill] \i circle(2pt);
	
	\draw[fill,lightgray] (-2,-1)--(-4,0)--(-6,-2)--(-4,-4)--(-2,-1);
	\draw (-2,-1)--(-4,0)--(-6,-2)--(-4,-4)--(-2,-1);
	
	\foreach \i in {(-1, 2), (0, 0), (-3, -1), (1, -5), (-1, -3), (-6, -2), (2, 3), (-2,
		2), (-3, 3), (0, 4), (1, -1), (-4, 0), (-5, 1), (2, -2), (-5, -4), (-4,
		-4), (2, 2), (3, -2), (-3, -2), (-2, -1), (-3, -6), (4, 1)}
	\draw \i circle(2pt);
	
	\draw (-5,1)--(-3,3)--(0,4)--(2,3)--(4,1)--(3,-2)--(1,-5)--(-3,-6)--(-5,-4)--(-6,-2)--(-5,1);
	
\node[left] at (-6,-2) {$a_0$};
	\node[below] at (0,-6) {$a_1=b_0$};
	\node[right] at (4,1) {$b_1$};
	\node at (-4.2,-2) {$B+a_0$};
	\node at (1.2,0) {$b_1+A$};
	\end{tikzpicture}

\end{center}
\caption{An illustration of the proof of Claim \ref{convex-position-case1}.}\label{fig:a0+B}
\end{figure}
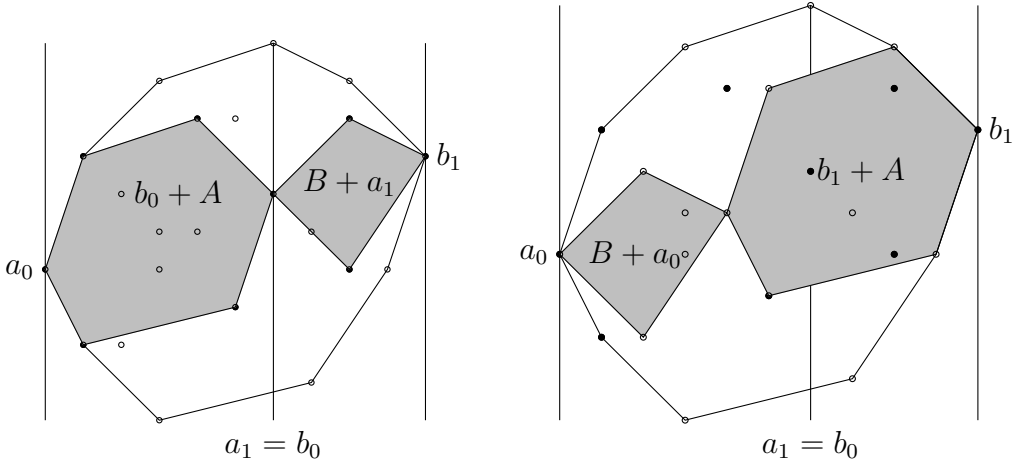
For (i), we note that if both
$x + b_0$ or $x + b_1$ are in $\partial [A+B]$, then they are the end points of a segment translated from $[b_0,b_1]$ and only two such translations have their end-points in $\partial [A+B]$ because $A$ and $B$ are not a strange pair. The argument for (ii) is similar.

Now (i) and (ii) say that counting the interior points of $(A + b_0) \cup (a_1 + B)$ and $ (a_0 + B) \cup (A + b_1)$ except $a_0+b_1$ and $a_1+b_0$ we have altogether at least $|\Delta_A| + |\Delta_B| - 8$ of them. Including the latter we have at least $|\Delta_A| + |\Delta_B| - 6$ of them and at least half of these in either $(A + b_0) \cup (a_1 + B)$ or $(a_0 + B) \cup (A + b_1)$, which yields \eqref{convex-position-case1}.\qed\end{proof}

Let us construct the suitable mixed triangulation of $[A+B]$. For every path $\sigma$  in $\partial A$, we assume that every point of $A$ in $\sigma$ is a vertex of $\sigma$. 
According to (\ref{convex-position-case1}), we may assume that
\begin{equation}
\label{convex-position-case10}
\left|(A \cup B)\cap{\rm int}[A+B]\right|\geq \frac{\Delta_A + \Delta_B} 2  - 3
\end{equation}
Let $a_{\rm upp}$ ($a_{\rm low}$) be the neighboring vertex of $[A]$ to $o$ on the upper (lower) arc of $\partial A$,
and let $b_{\rm upp}$ ($b_{\rm low}$) be the neighboring vertex of $[B]$ to $o$ on the upper (lower) arc of $\partial B$.
We write $\omega^A_{\rm upp}$ and $\omega^A_{\rm low}$ to denote the paths determined by $[o,a_{\rm upp}]$ and 
$[o,a_{\rm low}]$ and $\omega^B_{\rm upp}$ and $\omega^B_{\rm low}$ to denote the paths determined by 
$[o,b_{\rm upp}]$ and 
$[o,b_{\rm low}]$. Next let $\sigma^A_{\rm upp}$ ($\sigma^A_{\rm low}$) be the longest path
on the upper (lower) arc of $\partial[A]$ starting from $o$ such that every segment $s$ of $\sigma^A_{\rm upp}$ ($\sigma^A_{\rm low}$)  satisfies that $s+[o,b_{\rm upp}]$ ($s+[o,b_{\rm low}]$) is a parallelogram that does not intersect
${\rm int}[A]$. Similarly, let $\sigma^B_{\rm upp}$ ($\sigma^B_{\rm low}$) be the longest path
on the upper (lower) arc of $\partial[B]$ starting from $o$ such that every segment $s$ of $\sigma^B_{\rm upp}$ ($\sigma^B_{\rm low}$)  satisfies that $s+[o,a_{\rm upp}]$ ($s+[o,a_{\rm low}]$) is a parallelogram that does not intersect
${\rm int}[B]$. Since $a_1=b_0=o$ is a common point of $\sigma^A_{\rm upp}$, $\sigma^A_{\rm low}$,
 $\sigma^B_{\rm upp}$, $\sigma^B_{\rm low}$, we deduce from (\ref{convex-position-case10}) that
$$
1+(|\sigma^A_{\rm upp}|-1)+(|\sigma^A_{\rm low}|-1)+
 (|\sigma^B_{\rm upp}|-1)+(|\sigma^B_{\rm low}|-1)\geq \frac{\Delta_A + \Delta_B} 2  - 3,
$$
 equivalently,
\begin{equation}
\label{convex-position-case100}
|\sigma^A_{\rm upp}|+|\sigma^A_{\rm low}|+
 |\sigma^B_{\rm upp}|+|\sigma^B_{\rm low}|\geq \frac{\Delta_A + \Delta_B} 2 .
\end{equation}
We construct the mixed subdivision by considering the subdivisions into suitable paralleograms of 
$\sigma^A_{\rm upp}+\omega^B_{\rm upp}$ and $\sigma^B_{\rm upp}+\omega^A_{\rm upp}$
that have $\omega^A_{\rm upp}+\omega^B_{\rm upp}$ in common, and
the subdivisions into suitable parallelograms of 
$\sigma^A_{\rm low}+\omega^B_{\rm low}$ and $\sigma^B_{\rm low}+\omega^A_{\rm low}$
that have $\omega^A_{\rm low}+\omega^B_{\rm low}$ in common (see Figure \ref{fig:mixedconvex}).


\begin{figure}[h]
\begin{center}

	\begin{tikzpicture}[scale=0.55]
	
	
	\draw (-6,-6)--(-6,4) (0,-6)--(0,4) (4,-6)--(4,4);
	
	\foreach \i in {(0, 0),  (-2,2),(-5,1),(-6,-2),(-5,-4),(-1,-3)}
	\draw[fill] \i circle(2pt);
	
	\draw[fill, lightgray] (0, 0)--(-2,2)--(-5,1)--(-6,-2)--(-5,-4)--(-1,-3)--(0,0);
	\draw (0, 0)--(-2,2)--(-5,1)--(-6,-2)--(-5,-4)--(-1,-3)--(0,0);
	
	\foreach \i in {(0,0),(2,2),(4,1),(2,-2)}
	\draw[fill] \i circle(2pt);
	
	\draw[fill,lightgray] (0,0)--(2,2)--(4,1)--(2,-2)--(0,0);
	\draw (0,0)--(2,2)--(4,1)--(2,-2)--(0,0);
	
	\foreach \i in{(-1, 2), (0, 0), (-3, -1), (1, -5), (-1, -3), (-6, -2), (2, 3), (-2,
		2), (-3, 3), (0, 4), (1, -1), (-4, 0), (-5, 1), (2, -2), (-5, -4), (-4,
		-4), (2, 2), (3, -2), (-3, -2), (-2, -1), (-3, -6), (4, 1)}
	\draw \i circle(2pt);
	
	\draw (-5,1)--(-3,3)--(0,4)--(2,3)--(4,1)--(3,-2)--(1,-5)--(-3,-6)--(-5,-4);
	\draw[ultra thick]  (0,0)--(-2,2)  (0,0)--(-1,-3)  (0,0)--(2,2)  (2,-2)--(0,0);
	\draw[ultra thick,dashed] (-2,2)--(-5,1) (-1,-3)--(-5,-4)  (2,2)--(4,1)--(2,-2);
	\draw[pattern=north west lines, pattern color=lightgray]  (0,0)--(-2,2)--(0,4)--(2,2)--(0,0) (-2,2)--(0,4)--(-3,3)--(-5,1)--(-2,2) (0,4)--(2,2)--(4,1)--(2,3)--(0,4) (0,0)--(2,-2)--(1,-5)--(-1,-3)--(0,0) ;
	\draw[pattern=north east lines, pattern color=lightgray]  (0,0)--(-2,2)--(0,4)--(2,2)--(0,0)   (0,0)--(2,-2)--(1,-5)--(-1,-3)--(0,0) (1,-5)--(2,-2)--(4,1)--(3,-2)--(1,-5) (1,-5)-- (-1,-3)--(-5,-4)--(-3,-6)--(1,-5);
		\node[left] at (-6,-2) {$a_0$};
	\node[below] at (0,-6) {$a_1=b_0$};
	\node[right] at (4,1) {$b_1$};
	\node at (-2.5,0) {$A$};
	\node at (2,0) {$B$};
	\end{tikzpicture}
	
	\end{center}
\caption{An illustration of the parallelograms of the mixed subdivision in Case 1.}\label{fig:mixedconvex}
\end{figure}
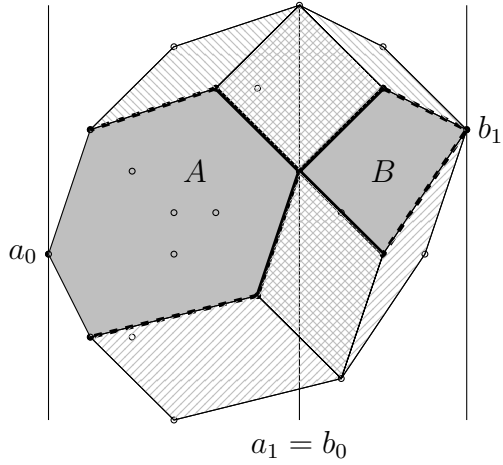
In particular,
$|\omega^A_{\rm upp}|,|\omega^B_{\rm upp}|\geq 1$,
(\ref{tsbig}) and (\ref{convex-position-case100}) yield that
\begin{eqnarray*}
|M_{11}|&\geq &
(|\sigma^A_{\rm upp}|-|\omega^A_{\rm upp}|)|\omega^B_{\rm upp}|+
(|\sigma^B_{\rm upp}|-|\omega^B_{\rm upp}|)|\omega^A_{\rm upp}|+
|\omega^A_{\rm upp}|\cdot |\omega^B_{\rm upp}|+\\
&&+
(|\sigma^A_{\rm low}|-|\omega^A_{\rm low}|)|\omega^B_{\rm low}|+
(|\sigma^B_{\rm low}|-|\omega^B_{\rm low}|)|\omega^A_{\rm low}|+
|\omega^A_{\rm low}|\cdot |\omega^B_{\rm low}|\\
&\geq &(|\sigma^A_{\rm upp}|-|\omega^A_{\rm upp}|)+
(|\sigma^B_{\rm upp}|-|\omega^B_{\rm upp}|)+
|\omega^A_{\rm upp}|+ |\omega^B_{\rm upp}|-1+\\
&&+
(|\sigma^A_{\rm low}|-|\omega^A_{\rm low}|)+
(|\sigma^B_{\rm low}|-|\omega^B_{\rm low}|)+
|\omega^A_{\rm low}|+ |\omega^B_{\rm low}|-1\\
&\geq &\frac{\Delta_A + \Delta_B}2-2=\frac{{\rm tr}(A) + {\rm tr}(B)} 2
\end{eqnarray*}
proving (\ref{convex-position-parallelograms}) in Case~1.\\

\noindent {\bf Case 2 } $A$ and $B$ form a strange pair with $|A|,|B|\geq 4$,
and $[A]$ and $[B]$ are not similar triangles 

We write $\alpha_{\rm upp}$ ($\alpha_{\rm low}$) to denote  the number of segments that the points of $A$ divide the upper (lower) arc of $\partial[A]$. We denote by $b_2$ the third vertex of $[B]$ and by $[x_0,x_1]$ the side of $A$  with $x_1-x_0=t(b_1-b_0)$ for $t>0$. 
For $i=0,1,2$, let $s_i$ be the number of segments that the points of $B$ divide the side of $[B]$ opposite to $b_i$.

\begin{claim}\label{claim:paths} There exists a $v$ such that  one of the following holds:
\begin{align}
&\alpha_{\rm upp}\geq 2 \; \text{and} \; \alpha_{\rm upp}+s_0+s_1\geq \frac12(\Delta_A+\Delta_B), \text{or} \label{convex-position-case21}\\
&\alpha_{\rm low},s_2\geq 2 \;  \text{and} \; 
\alpha_{\rm low}+s_2\geq \frac12(\Delta_A+\Delta_B).\label{convex-position-case22}
\end{align}
\end{claim}

\begin{proof}  Since $\alpha_{\rm upp}+\alpha_{\rm low}=\Delta_A$ and
$s_0+s_1+s_2=\Delta_B$, the claim easily follows if there is a  $v$  such that,  for each  the sets $A$ and $B$, both the upper arc and the lower arc contain a point of the set strictly between the two supporting lines parallel to $v$. 

Otherwise, choose a $v$ such  that the side $[b_0,b_1]$  of $[B]$ contains at least $3$ points of $B$ (this is possible since $|B|\ge 4$). Then  $[x_0,x_1]$ has no other point of $A$  than $x_0,x_1$ and the other side of $[A]$  at $x_i$, $i=0,1$ is parallel to $[b_i,b_2]$. As $[A]$ and $[B]$ are not similar triangles , $[A]$ has some more sides, which in turn yields that $[b_i,b_2]\cap B=\{b_i,b_2\}$ for $i=0,1$. 
In summary, we have 
$\alpha_{\rm upp}=s_0=s_1=1$ and $\alpha_{\rm low},s_2\geq 2$. Since
$\alpha_{\rm low}+s_2>\alpha_{\rm upp}+s_0+s_1$, we conclude (\ref{convex-position-case22}).\qed\end{proof}

To prove (\ref{convex-position-parallelograms}) based on (\ref{convex-position-case21}) and (\ref{convex-position-case22}), we introduce some further notation. After a linear transformation, we may assume that $v$ is an exterior normal to the side $[b_0,b_1]$ of $[B]$.
We say that $p,q\in \partial[A]$ are opposite if there exists a unit vector $w$ such that $w$ is an exterior normal at $p$ and $-w$ is an exterior normal at $q$. If $p,q\in \partial[A]$ are not opposite, then we write
$\overline{pq}$ the arc of $\partial[A]$ connecting $p$ and $q$ and not containing opposite pair of points. 

First we assume that (\ref{convex-position-case21}) holds and $b_2=o$. Since $[x_0,x_1]$ has  exterior normal $v$ and 
$\alpha_{\rm upp}\geq 2$, there exists
$a\in A\backslash\{x_0,x_1\}$ such that $v$ is an exterior normal to $\partial[A]$ at $a$. We write 
$l_{\rm upp}$ and $r_{\rm upp}$ to denote the number of segments the points of $A$ divide the arcs
$\overline{ax_0}$ and $\overline{ax_1}$, respectively. To construct a mixed subdivision, we observe that
every exterior normal $u$ to a side of $[A]$ in $\overline{ax_0}$ satisfies $\langle u,b_0\rangle>0$, and
every exterior normal $w$ to a side of $[A]$ in $\overline{ax_1}$ satisfies $\langle w,b_1\rangle>0$.
We divide
$\overline{ax_0}+[o,b_0]$ into suitable $s_1l_{\rm upp}$ parallelograms, and
$\overline{ax_1}+[o,b_1]$ into suitable $s_0r_{\rm upp}$ parallelograms. It follows from (\ref{tsbig}) that
\begin{eqnarray*}
|M_{11}|&=&s_1l_{\rm upp}+s_0r_{\rm upp}\geq l_{\rm upp}+r_{\rm upp}+s_0+s_1-2=
\alpha_{\rm upp}+s_0+s_1-2\\
&\geq&
\mbox{$\frac12(\Delta_A+\Delta_B)-2=\frac12({\rm tr}(A)+{\rm tr}(B))$}.
\end{eqnarray*}

Secondly we assume that (\ref{convex-position-case22}) holds. Since $s_2\geq 2$, we may assume that
  $o\in ([b_0,b_1]\backslash\{b_0,b_1\})\cap B$. 
For $i=0,1$, we write $s_{2i}$ to denote the number of segments the points of $B$ divide $[o,b_i]$.
Let $\tilde{x}_0$ and $\tilde{x}_1$ be the leftmost and rightmost points of $A$ such that $-v$ is an exterior normal to $\partial[A]$, where possibly $\tilde{x}_0=\tilde{x}_1$. 
Since $[A]$ has  sides parallel to the sides $[b_2,b_0]$ and 
$[b_2,b_1]$ of $[B]$, we deduce that $\tilde{x}_0\neq x_0$ and $\tilde{x}_1\neq x_1$.
To construct a mixed subdivision, we set $m_{\rm low}=0$ if $\tilde{x}_0=\tilde{x}_1$, and
$m_{\rm low}$ to be the number of segments the points of $A$ divide $\overline{\tilde{x}_0,\tilde{x}_1}$
 if $\tilde{x}_0\neq\tilde{x}_1$.
In addition, we write 
$l_{\rm low}\geq 1$ and $r_{\rm low}\geq 1$ to denote the number of segments the points of $A$ divide the arcs
$\overline{\tilde{x}_0x_0}$ and $\overline{\tilde{x}_1x_1}$, respectively.
We divide
$\overline{\tilde{x}_0x_0}+[o,b_0]$ into suitable $l_{\rm low}s_{20}$ parallelograms, and
$\overline{\tilde{x}_1x_1}+[o,b_1]$ into suitable $r_{\rm upp}s_{21}$ parallelograms. In addition, 
 if $\tilde{x}_0\neq\tilde{x}_1$, then we divide 
$[\tilde{x}_0\tilde{x}_1]+[o,b_2]$ into suitable $m_{\rm low}$ parallelograms.
It follows from (\ref{tsbig}) that
\begin{eqnarray*}
|M_{11}|&=&l_{\rm low}s_{20}+r_{\rm low}s_{21}+m_{\rm low}\geq
 l_{\rm low}+r_{\rm low}+m_{\rm low}+s_{20}+s_{21}-2\\
&=& \alpha_{\rm low}+s_2-2
\geq
\mbox{$\frac12(\Delta_A+\Delta_B)-2=\frac12({\rm tr}(A)+{\rm tr}(B))$},
\end{eqnarray*}
finishing the proof of (\ref{convex-position-parallelograms}) in Case~2.\\

\noindent {\bf Case 3 } $[A]$ and $[B]$ are similar triangles and $|A|,|B|\geq 4$

We recall that $s_1, s_2$ and  $s_3$ denote the number of segments the points of $B$ divide the sides of $[B]$ and let $s'_1, s'_2, s'_3$ be the number of segments the points of $A$ divide the corresponding sides of $[A]$. 
We have ${\rm tr}(A)=s_1'+s'_2+s'_3-2$ and ${\rm tr}(B)=s_1+s_2+s_3-2$. We may assume that $s_1$ is the largest among 
the six numbers and that $s'_2\geq s'_3$. Readily
\begin{equation}
\label{m11case3}
|M_{11}|\geq \max\{s_1s_2', s'_1s_2, s'_1s_3\}.
\end{equation}
If $s_2'\geq 3$, then
$$
|M_{11}|\geq 3s_1\geq \mbox{$\frac12$}(s_1+s_2+s_3+s'_1+s'_2+s'_3)>\mbox{$\frac12$}({\rm tr}(A)+{\rm tr}(B)).
$$
If $s_2'=2$, then $s_3'\leq 2$ and
$$
|M_{11}|\geq 2s_1\geq \mbox{$\frac12$}(s_1+s_2+s_3+s'_1+s'_2+s'_3-4)=\mbox{$\frac12$}({\rm tr}(A)+{\rm tr}(B)).
$$
Therefore we assume that $s_2'=s_3'=1$. In particular, we may also assume that $s_2\geq s_3$. Since $s'_1\ge 2$ and $s_2\ge 1$ we have $s'_1s_2\ge s'_1+2s_2-2$. Therefore,
\begin{align*}
|M_{11}|&\ge \max \{s_1, s'_1s_2\}\\
&\ge \frac{1}{2}((s_1+s_2+s_3+s'_1-2)\\
&\ge  \frac{1}{2}(s_1+s_2+s_3+s'_1-2)\\
&=\frac{1}{2}({\rm tr}(A)+{\rm tr}(B)),
\end{align*}
and  we conclude (\ref{convex-position-parallelograms}) in Case~3, as well.
\qed

\section{Proof of Theorem~\ref{TA=TB}}

Let $A=\{a_1,\ldots,a_n\}$. Naturally, $[A+A]$ has a triangulation $\{F+F:\,F\in T_A\}$, which we subdive in order to obtain $M$. We define $M$ to be the collection of the sums of the form
$$
[a_{i_0},\ldots,a_{i_m}]+[a_{i_m},\ldots,a_{i_k}]
$$
where $k\geq 0$, $0\leq m\leq k$, $i_j<i_l$ for $j<l$,  and $[a_{i_0},\ldots,a_{i_k}]\in T_A$.

To show that we obtain a cell decomposition, let
$$
F=[a_{i_0},\ldots,a_{i_k}]\in T_A
$$
be a $k$-simplex with $k>0$ where $i_j<i_l$ for $j<l$, and hence
$$
F+F=\left\{\sum_{i=0}^k\alpha_ja_{i_j}:\,\sum_{i=0}^k\alpha_j=2\;\&\;\forall\,\alpha_j\geq 0\right\}.
$$
We write ${\rm relint}\,C$ to denote the relative interior of a compact convex set $C$. For some $0\leq m\leq k$, 
$\alpha_0,\ldots,\alpha_k\geq 0$ with $\sum_{i=0}^k\alpha_j=2$, we have
$$
\sum_{i=0}^k\alpha_ja_{i_j}\in{\rm relint}\,\left([a_{i_0},\ldots,a_{i_m}]+[a_{i_m},\ldots,a_{i_k}]\right)\subset F+F
$$
if and only if $\sum_{j<m}\alpha_j<1$ and $\sum_{i=0}^m\alpha_j>1$ where we set
$\sum_{j<0}\alpha_j=0$. 
We conclude that $M$ forms a cell decomposition of $[A+A]$.

 For any $d$-simplex $F\in T_A$, and for any $m=0,\ldots,d$, we have constructed one $d$-cell of $M$ that is the sum of an $m$-simplex and a $(d-m)$-simplex. Therefore
$$
\|M\|=|T_A|\sum_{m=0}^d{d \choose m}=2^d|T_A|.
$$

\section{Proof of Corollary~\ref{A=Bstability}}

In this section, let $A\subset \R^2$ be finite and not collinear. We prove four auxiliary statements about $A$. The first is an application of the case $A=B$ of Conjecture~\ref{ruzsabrunnconj} (see Theorem~\ref{A=B}).

\begin{lemma}
\label{|A+A|}
$$
|A+A|\geq 4|A|-\Delta_A-3
$$
\end{lemma}
\proof We have readily $\Delta_{A+A}\geq 2\Delta_A$. Thus  (\ref{Eulerpoints}) and Theorem~\ref{A=B} yield
$$
|A+A|=\frac12\left({\rm tr}(A+A)+\Delta_{A+A}+2\right)\geq
2{\rm tr}(A)+\Delta_{A}+1=4|A|-\Delta_A-3.\qed$$

We note that the estimate of Lemma~\ref{|A+A|} is optimal, the configuration of Theorem \ref{A=B} (b) being an extremal set.

Next we provide the well-known elementary estimate for $|A+A|$ only in terms of boundary points. 

\begin{lemma} Let $m_A$  denote the maximal number of points of $A$ contained in a side of $[A]$. We have,
\label{|A+A|boundary}
$$
|A+A|\geq \frac{\Delta_A^2}4-\frac{\Delta_A(m_A-1)}2.
$$
\end{lemma}
\proof  We choose a line $l$ not parallel to any side of $[A]$, that we may assume to be a vertical line, and denote by $s_1,\ldots ,s_k$ the sides of $[A]$ on the upper chain of $[A]$ in left to right order. Let $A_i$ be the set obtained from $A\cap s_i $ by removing its rightmost point. We may assume that
$$
|A_1|+\cdots +|A_k|\ge \frac{\Delta_A}{2}.
$$
We observe that, for $1\le i<j\le k$, we have
$$
|A_i+A_j|=|A_i|\cdot |A_j|\; \text{and}\;  (A_i+A_j)\cap (A_{i'}\cap A_{j'})=\emptyset \; \text{if}\;  \{i,j\}\neq \{i',j'\}.
$$ 
It follows that
\begin{align*}
|A+A|&\ge \sum_{1\le i<j\le k} |A_i+A_j|=\sum_{1\le i<j\le k} |A_i|\cdot |A_j|=(\sum_{i=1}^k |A_i|)^2-\sum_{i=1}^k |A_i|^2\\
&\ge \left(\frac{\Delta_A}{2}\right)^2-(m_A-1)\frac{\Delta_A}{2}.\qed
\end{align*}

The following Lemma can be found in Freiman \cite{Fre73}.

\begin{lemma}\label{scover} Let $\ell$ be a line intersecting $[A]$ in $m$ points of $A$. If $A$ is covered by exactly $s$  lines parallel to $\ell$, then
\begin{equation}\label{eq:weak4s}
|A+A|\geq 2|A|+(s-1)m-s.
\end{equation}
Moreover,  
\begin{equation}\label{eq:4s}
|A+A|\ge (4-\frac{2}{s})|A|-(2s-1).
\end{equation}
\end{lemma} 

\begin{proof} We may assume that  $\ell$ is the vertical line through the origin, that $a_1,\ldots ,a_s$ are $s$ points of $A$ ordered left to right  such that $A=\cup_{i=1}^s (A\cap (\ell +a_i))$ and that $|A\cap (\ell +a_1)|=m$. Let  $A_i=A\cap (a_i+\ell)$. Then,
\begin{align*}
|A+A|&= |A_1+A|+|(A\setminus A_1)+A_s|\\
&\ge \sum_{i=1}^s (|A_1|+|A_i|-1)+\sum_{i=2}^s (|A_i|+|A_s|-1)\\
&=2|A|+(s-1)(|A_1|+|A_s|)-(2s-1),
\end{align*}
from which \eqref{eq:weak4s} follows. On the other hand,
\begin{align*}
|A+A|&=\sum_{i=1}^s |2A_i|+\sum_{i=1}^{s-1} |A_i+A_{i+1}|\\
&\ge \sum_{i=1}^s (2|A_i|-1)+\sum_{�=1}^{s-1} (|A_i|+|A_{i+1}|-1)\\
&=4|A|-(|A_1|+|A_s|)-(2s-1).
\end{align*}
If the latter estimate is larger than the former one we obtain \eqref{eq:4s}, otherwise we get the stronger inequality $|A+A|\ge (4-2/s^2)|A|-(2s-1)$.
\qed\end{proof}

\noindent{\bf Proof of Corollary~\ref{A=Bstability} } Let $|A+A|\leq (4-\varepsilon)|A|$ where $\varepsilon\in(0,1)$
and $\varepsilon^2 |A|\geq 48$. 
To simply formulae, we set $\Delta=\Delta_A$ and $m=m_A$.

We deduce from Lemma~\ref{|A+A|} that $\Delta\geq \varepsilon |A|-3$. Substituting this into 
Lemma~\ref{|A+A|boundary} yields
\begin{eqnarray*}
(4-\varepsilon)|A|&\geq& \frac{\Delta^2}4-\frac{\Delta(m-1)}2\geq
 \frac{\Delta(\varepsilon |A|-3)}4-\frac{\Delta(m-1)}2\\
&=&\frac{\Delta}2\cdot (\mbox{$\frac12\varepsilon |A|-m-\frac12$})\geq 
\frac{\varepsilon |A|-3}2\cdot (\mbox{$\frac12\varepsilon |A|-m-\frac12$}).
\end{eqnarray*}
Therefore
$$
\mbox{$\frac12\varepsilon |A|-(m-1)$}\leq \frac{8}{\varepsilon}\left(1-\frac{\varepsilon}4\right)
\left(1+\frac{3}{\varepsilon |A|-3}\right)+\frac32< \frac{12}{\varepsilon}
$$
as $\varepsilon |A|-3\geq \frac{48}{\varepsilon}-3> \frac{12}{\varepsilon}$.
In particular, $m-1>\frac12\varepsilon |A|-\frac{12}{\varepsilon}$.

Next let $l$ be the line determined by a side of $[A]$ containing $m=m_A$ point of $A$, and let $s$ be the number of lines
parallel to $l$ intersecting $A$. According to \eqref{eq:weak4s},
$$
(4-\varepsilon)|A|\geq 2|A|+(s-1)(m-1)-1> 2|A|+(s-1)\mbox{$(\frac12\varepsilon |A|-\frac{12}{\varepsilon})$}-1,
$$
thus first rearranging, and then applying $\varepsilon^2 |A|\geq 48$ yield
$$
2|A|> s\cdot \mbox{$(\frac12\varepsilon |A|-\frac{12}{\varepsilon})$}\geq  s\cdot \mbox{$\frac14\varepsilon |A|$}.
$$
Therefore $s<\frac8\varepsilon$.

We deduce from \eqref{eq:4s}  and $s<\frac8\varepsilon$ that
$$
(4-\varepsilon)|A|> \mbox{$(4-\frac2s)$}|A|-2s>
\mbox{$(4-\frac2s)|A|-\frac{16}{\varepsilon}$}.
$$
Rearranging, and then applying $\varepsilon^2 |A|\geq 48$ imply
$$
s<\frac2\varepsilon\left(1-\frac{16}{\varepsilon^2|A|}\right)^{-1}<
\frac2\varepsilon\left(1+\frac{32}{\varepsilon^2|A|}\right).
\qed$$

\section{Proof of Proposition~\ref{counterexample} }
%

We call  the points of $A$, 
$$
a_0=(0,0),\mbox{ \ }a_1=(-1,-2),\mbox{ \ }a_2=(2,1).
$$

If $k\geq 2$, then we show that every mixed subdivision $M$ corresponding to $T_A$ and $T_B$ satisfies
\begin{equation}
\label{counter11}
|M_{11}|\leq 24.
\end{equation}

We prove (\ref{counter11}) in several steps. First we verify
\begin{eqnarray}
\label{lia1a2}
[a_1,a_2]+l_i&\mbox{ is not an edge of $M$ }& \mbox{ for }i=0,\ldots,k\\
\label{ria1a2}
[a_1,a_2]+r_i&\mbox{ is not an edge of $M$ }&\mbox{  for }i=0,\ldots,k-1.
\end{eqnarray}
For (\ref{lia1a2}), we observe that $a_1+l_{i+1}$ if $i\leq k-1$ or $a_1+l_{i-1}$ if $i\geq 1$  is a point of $A+B$ in 
$[a_1,a_2]+l_i$ different from the endpoints. Similarly, for (\ref{ria1a2}), we observe that $a_1+r_{i+1}$ if $i\leq k-2$ or $a_1+r_{i-1}$ if $i\geq 1$  is a point of $A+B$ in 
$[a_1,a_2]+r_i$ different from the endpoints.

Next, we have
\begin{eqnarray}
\label{liria0a2}
[a_0,a_2]+[l_i,r_i]&\mbox{is not a parallelogram of $M$}& \mbox{ for }i=0,\ldots,k-1\\
\label{li+1ria0a1}
[a_0,a_1]+[r_i,l_{i+1}]&\mbox{is not a parallelogram of $M$}&\mbox{  for }i=0,\ldots,k-1
\end{eqnarray}
as $l_{i+1}\in{\rm int}\,[a_0,a_2]+[l_i,r_i]$ and $l_{i}\in{\rm int}\,[a_0,a_1]+[r_i,l_{i+1}]$.

Let us call the edges of $T_B$ of the form either $[l_i,r_i]$ or $[r_i,l_{i+1}]$  for $i=0,\ldots,k-1$ {\it small edges}, and the edges
of $T_B$ of the form either $[p,l_i]$, $[q,l_i]$ for $i=0,\ldots,k$, or $[p,r_i]$, $[q,r_i]$  for $i=0,\ldots,k-1$ {\it long edges}. In other words, long edges of $T_B$ contain either $p$ or $q$, while small edges of $T_B$ contain neither.

Concerning long edges, we prove that that the number of parallelograms   of $M$ of the form
\begin{equation}
\label{longedge}
\mbox{$e_A+e_B$ for an edge $e_A$ of $T_A$ and a long edge $e_B$ of $T_B$ is at most $12$.}
\end{equation} 
If  $e_A$ is an edge of $T_A$, then there exist at most two cells of $M$ whose side are $p+e_A$. Since $T_A$ has three edges, there are at most six of parallelograms  of $M$ of the form
$e_A+e_B$ where $e_A$ is an edge  of $T_A$ and $e_B$ is an edge of $T_B$ with $p\in e_B$. Since the same estimate holds if $q\in e_B$, we conclude (\ref{longedge}).

Finally,  we prove that that the number of parallelograms   of $M$ of the form
\begin{equation}
\label{smalledge}
\mbox{$e_A+e_B$ for an edge $e_A$ of $T_A$ and a small edge $e_B$ of $T_B$ is at most $12$.}
\end{equation} 
The argument for (\ref{smalledge}) is based on the claim that if
$e_A+e_B$ is a  parallelogram of $M$  for an edge $e_A$ of $T_A$ and a small edge $e_B$ of $T_B$, then there is
a long edge $e'_B$ of $T_B$ such that
\begin{equation}
\label{smalledge0}
\mbox{$e_A+e'_B$ is a neighboring parallelogram of $M$.}
\end{equation} 
 We have $e_A\neq[a_1,a_2]$ according to  (\ref{lia1a2})  and (\ref{ria1a2}). If
$e_A=[a_0,a_1]$, then $e_B=[l_i,r_i]$ for some $i\in\{1,\ldots,k-1\}$ according to      
(\ref{li+1ria0a1}). Now $r_i+e_A$ intersects the interior of $[A+B]$ as $r_i\in{\rm int}\,[A]$, thus
it is the edge of another cell of $M$, as well. This other cell is either a translate of $[A]$, which is impossible by                           
(\ref{lia1a2}), (\ref{ria1a2}), and as $r_i\not\in p+[A],q+[A]$, or of the form $e_A+e'_B$ for an edge $e'_B\neq e_B$
of $T_B$ containing $r_i$. However, $e'_B\neq[r_i,l_{i+1}]$ by (\ref{li+1ria0a1}), therefore $e'_B$ is a long edge.

On the other hand, if $e_A=[a_0,a_2]$, then $e_B=[r_i,l_{i+1}]$ for some $i\in\{1,\ldots,k-1\}$ according to      
(\ref{liria0a2}), and (\ref{smalledge0}) follows as above.

Now if $e_A+e'_B$ is a parallelogram of $M$ for an edge $e_A$ of $T_A$ and a long edge $e'_B$ of $T_B$,
then there is at most one neighboring paralellogram of the form $e_A+e_B$ for a small edge $e_B$ of $T_B$ because 
$e_A+e_B$ does not intersect $e_A+p$ and $e_A+q$. In turn, (\ref{smalledge}) follows from 
(\ref{longedge}) and (\ref{smalledge0}). Moreover, we conclude (\ref{counter11}) from
(\ref{longedge}) and (\ref{smalledge}).

Finally, it follows from (\ref{counter11}) that if $k\geq 145$, then
$$
|M_{11}|\leq 24<\sqrt{4k}=\sqrt{|T_A|\cdot|T_B|}.\qed
$$

\end{document}